\documentclass[11pt]{amsart}
\usepackage{mathptmx}
\usepackage{amsmath,amsthm}
\usepackage{amscd,tikz,mathtools,multicol}
\pdfoutput=1
\usepackage{amsfonts,enumerate}
\usepackage{amssymb}
\usepackage{graphicx,geometry,cite,adjustbox}
\usepackage{mathrsfs}
\usepackage[all,cmtip,2cell]{xy}
\usepackage[pagebackref,colorlinks]{hyperref}
\hypersetup{pdfborder={0 0 0}}

\renewcommand*{\backref}[1]{}
\renewcommand*{\backrefalt}[4]{%
	\ifcase #1 %
	No citations.
	\or
	(Page #2).%
	\else
	(Pages #2).%
	\fi%
}

\UseTwocells
\xyoption{poly}

\newcommand{\lel}{\mathrel{\leqslant_\ell}}
\newcommand{\ler}{\mathrel{\leqslant_{r}}}
\newcommand{\gl}{\mathrel{\mathscr{L}}}
\newcommand{\gr}{\mathrel{\mathscr{R}}}

\newcommand{\gd}{\mathrel{\mathscr{D}}}
\newcommand{\snc}{\widehat{\mathcal{C}}}
\newcommand{\cod}{\mathcal{C}_\mathfrak{D}}
\newcommand{\sncd}{\widehat{\mathcal{C}_\mathfrak{D}}}
\newcommand{\sncdo}{\widehat{\mathcal{C}_\mathfrak{D}^\text{op}}}
\newcommand{\codp}{\mathcal{C'}_\mathfrak{D'}}

\newcommand{\sncdp}{\widehat{\mathcal{C'}_\mathfrak{D'}}}
\newcommand{\lr}{\mathbb{L}(S)_\mathfrak{R}}
\newcommand{\rl}{\mathbb{R}(S)_\mathfrak{L}}
\newcommand{\lrp}{\mathbb{L}(S')_\mathfrak{R'}}
\newcommand{\rlp}{\mathbb{R}(S')_\mathfrak{L}}
\newcommand{\lrb}{\mathbb{L}(B)_\mathfrak{R}}
\newcommand{\sls}{\widehat{\mathbb{L}(S)}}
\newcommand{\slr}{\widehat{\mathbb{L}(S)_\mathfrak{R}}}
\newcommand{\srlo}{\widehat{\mathbb{R}(S)_\mathfrak{L}^\text{op}}}
\newcommand{\srs}{\widehat{\mathbb{R}(S)}}
\newcommand{\srso}{\widehat{\mathbb{R}(S)}^\text{op}}

\newcommand{\los}{\mathbb{L}(S)}
\newcommand{\ros}{\mathbb{R}(S)}
\newcommand{\losp}{\mathbb{L}(S')}
\newcommand{\sts}{\mathscr{T}_n\backslash \mathscr{S}_n}
\newcommand{\pn}{\mathbb{SP}}
\newcommand{\fpn}{\mathbb{P}}
\newcommand{\sfpn}{\widehat{\fpn}}
\newcommand{\ns}{\mathbf{n}}
\newcommand{\tn}{\mathscr{T}_n}
\newcommand{\ix}{\mathscr{I}_X}
\newcommand{\pin}{\text{N}\Pi}
\newcommand{\tv}{\mathscr{L\mspace{-5mu}T}_V}
\newcommand{\glv}{\mathscr{G\mspace{-5mu}L}_V}
\newcommand{\stv}{\mathscr{L\mspace{-5mu}T}_V \backslash\mathscr{G\mspace{-5mu}L}_V}
\newcommand{\sv}{\mathbb{V}}
\newcommand{\psv}{\mathbb{SV}}
\newcommand{\x}{\mathbb{X}}

\newtheorem{thm}{Theorem}[section]
\newtheorem{lem}[thm]{Lemma}
\newtheorem{pro}[thm]{Proposition}
\newtheorem{cor}[thm]{Corollary}
\newtheorem{dfn}[thm]{Definition}
\newtheorem{rmk}[thm]{Remark}
\newtheorem{ex}[thm]{Example}

\DeclareMathOperator{\im}{im}
\DeclareMathOperator{\coim}{coim}

\title[Left reductive regular semigroups]{Left reductive regular semigroups}

\author{P. A. Azeef~Muhammed}
\address{Centre for Research in Mathematics and Data Science, Western Sydney University, Locked Bag 1797, Penrith NSW 2751, Australia.}
\email{azeefp@gmail.com, A.ParayilAjmal@WesternSydney.edu.au}
\author{Gracinda M. S. Gomes}
\address{Faculdade de Ci\^{e}ncias, Departamento de Matem\'{a}tica, Universidade de Lisboa, 1749-016, Lisboa, Portugal}
\email{gracindamsgmcunha@gmail.com, gmcunha@fc.ul.pt}
\subjclass[2010]{}
\thanks{This work was partially developed within the activities of Departamento de Matemática da Faculdade de Ciências da 	Universidade de Lisboa and 	Centro de Matemática Computacional e Estocástica, CEMAT, within the projects
	UIDB/04621/2020 and UIDP/04621/2020, financed by Fundação para	a Ciência e a Tecnologia, FCT}

\begin{document}
	\maketitle
	\begin{abstract}
		
		In this paper we develop an ideal structure theory for the class of left reductive regular semigroups and apply it to several subclasses of popular interest. In these classes we observe that the right ideal structure of the semigroup is `embedded' inside the left ideal one, and so we can construct these semigroups starting with only one object (unlike in other more general cases). To this end, we introduce an upgraded version of Nambooripad's normal category \cite{cross} as our building block, which we call a \emph{connected category}.
		
		The main theorem of the paper describes a category equivalence between the category of left (and right) reductive regular semigroups and the category of {connected categories}. Then, we specialise our result to describe constructions of $\gl$- (and $\gr$-) unipotent semigroups, right (and left) regular bands, inverse semigroups and arbitrary regular monoids. Finally, we provide concrete (and rather simple) descriptions to the connected categories that arise from finite transformation semigroups, linear transformation semigroups (over a finite dimensional vector space) and symmetric inverse monoids. 
	\end{abstract}
	
	\tableofcontents
	
	\section{Introduction}\label{secintro}
	
	The most important algebraic invariants of any given semigroup are its Green's relations which describe the ideal structure of the semigroup. Introduced in a seminal paper \cite{green1951} in 1951 Green's relations are certain equivalence relations defined on a semigroup which partition the semigroup elements into an `egg-box' diagram (for example, see Figures \ref{figdeco}, \ref{figcomp} and \ref{figinv}). In this partitioning, the elements generating the same principal left ideal fall in the same column of the egg-box and  those generating the same principal right ideal fall in the same row. This captures a lot of information regarding the local and global structure of the semigroup. In fact, it is precisely this partitioning that makes semigroups manageable, in spite of them being rather general objects. Hence, it is no surprise that  Howie would remark that the Green's relations are ``so all-pervading that, on encountering a new semigroup, almost the first question one asks is `What are the Green relations like?'" \cite{howie2002}. Therefore, it is very natural that any structure theorem for semigroups may aim to begin with building blocks that abstract the principal left (and right) ideals. To make any headway in this direction, one needs to understand the inter-relationship between principal left and principal right ideal structures. In general, this is quite complicated and invariably involves two ordered sets, each representing the left and right ideals of the semigroup, interconnected in a non-trivial fashion. Unsurprisingly, this rather difficult question is still open in many general cases. Such a quest leads naturally towards the special class of regular semigroups. 
	
	In regular semigroups, each principal left (and right) ideal is generated by idempotent elements, giving some control over the structure of the semigroup. Indeed there is a very close relationship between the ideal structure and the idempotent structure and, in fact, we can obtain one from the other \cite{bicxn,indcxn1,indcxn2}. Recall that regular semigroups were introduced by Green in \cite{green1951}, wherein he credited Rees for the suggestion to adopt von Neumann's definition \cite{rr} from ring theory.
	\begin{dfn}
		A semigroup $S$ is said to be (von Neumann) \emph{regular} if for every element $a$ in $S$, there exists $x\in S$ such that $axa=a$.
	\end{dfn}

	Historically, one of the first major leaps into regular semigroups was by Hall \cite{hall} who extended Munn's \cite{munn} construction of \emph{fundamental inverse semigroups} to fundamental regular semigroups generated by idempotents. To this end, Hall considered certain transformations on the partially ordered set (poset) of the principal left ideals and on that of the principal right ideals. Later, Grillet \cite{grilcross} gave an abstract characterisation of these posets as \emph{regular posets}, and introduced the notion of \emph{cross-connection} to describe the exact relationship between the left and the right posets of a regular semigroup. Simultaneously, Nambooripad \cite{ksssf,ksssf1,mem} developed the idea of (regular) biordered set, as an abstract model of the set of idempotents of a (regular) semigroup, and using groupoids gave a general structure theorem for regular semigroups. This seminal work also described an equivalence between the category of regular semigroups and the category of certain {groupoids} and, in the process, it puts final touches to the celebrated Ehresmann-Schein-Nambooripad (ESN) Theorem \cite{lawson}. Although extremely clever, Nambooripad's description of the sets of idempotents as biordered sets is still complicated and pretty cumbersome to work with, especially for constructions. It may be worth mentioning that a practical way around this is to work with an associated semigroup since (regular) biordered sets come from (regular) semigroups \cite{mem,eas}.
	
	In 1978, Nambooripad \cite{bicxn} showed that regular biordered sets and cross-connected regular posets are equivalent. Elaborating on this fact, he developed his theory of cross-connections \cite{cross} by replacing regular posets with \emph{normal categories} (which contain regular posets as subcategories). In this way, he proved that the category of regular semigroups is equivalent to the category of cross-connected normal categories.
	
	{A major problem with such a general approach  is that the theory developed is too heavy to be applied to the vast majority of the objects which may have a rather simple structure!} In fact, we  believe this is one of the major reasons why Nambooripad's cross-connection theory has not achieved the popularity and acclaim that such a deep work deserves. Addressing this wide gap in the development of the area is one of the main motivations behind this paper.
	
	As the reader may soon observe, the entire discussion in this paper can be traced back to Cayley's theorem for groups. Just as any group $G$ may be realised as a subgroup of the symmetric group on the set $G$, it is well-known that any semigroup $S$ can be realised as a transformation semigroup on the set $S^1:=S\cup{1}$, \cite{howie}. This may be achieved by considering  the  \emph{regular representation}   of  $S$ \cite[Section 1.3]{clif},  which is the homomorphism $\rho\colon S \mapsto \mathscr{T}_S$, $a\mapsto\rho_a$ , where   $\rho_a\colon x\mapsto xa$, for every $x\in S$. 
	Adjoining the element $1$ to the set $S$ (if $S$ is not a monoid) is sufficient to ensure injectivity of the representation, and so in this case $S$ is isomorphic $S\rho$.
	
	So, a question arises: for which classes of semigroups do we have injectivity without adjoining $1$? This leads us to left reductive semigroups. 
	
	\begin{dfn}
		A semigroup $S$ is said to be \emph{left reductive} if the regular representation $\rho$ is injective.	
	\end{dfn}
	
	In this paper we discuss the ideal structure of the class of left reductive regular semigroups 	and apply our structural result to obtain constructions for several popular subclasses. Left reductive regular semigroups include in particular, all regular monoids, $\gl$-unipotent semigroups, inverse semigroups, right regular bands, full (linear) transformation semigroups, singular transformation  semigroups on a finite set, and the semigroup of singular linear transformations over a finite dimensional vector space. Each of these classes arises naturally across mathematics, statistics, and physics. We will examine each class of semigroups in detail. Observe that left reductive regular semigroups exclude several `simple' classes like completely simple semigroups \cite{css}, bands \cite{preenu22} and regular-$*$ semigroups \cite{eastazeef}. 
	
	In left reductive regular semigroups, the relationship between the left and the right ideals is relatively simpler and more transparent than in arbitrary regular semigroups. Roughly speaking, in the left reductive regular case, the left and the right ideals are very tightly interconnected, and the $\gl$-structure totally restricts the $\gr$-structure. More precisely, given a left reductive regular semigroup $S$, it determines two ordered objects:  
	\begin{enumerate}[(i)]
		\item a category $\los$ of principal left ideals (see (\ref{los}-\ref{los1})), and 
		\item a down-set  $\mathfrak{R}$ of the poset of principal right ideals of the semigroup $\sls$ that arises from the category $\los$ (see (\ref{eqnr})).		
	\end{enumerate}	 
	We shall show that  (see Proposition \ref{prolsr}) the poset of principal right ideals of $S$ is isomorphic to the down-set  $\mathfrak{R}$, which in turn `sits inside' the category $\los$ described in (i) above. As a result of this rather neat situation in left reductive regular semigroups, we can avoid the use of several complicated notions that would otherwise be required in a general discussion of arbitrary regular semigroups (see \cite{cross}).
	
	Just as an element of a group $G$ may be represented as a permutation of the set $G$, in this paper we represent an element of a left reductive regular semigroup $S$ by a \emph{cone} in the category $\los$ (see Definition \ref{dfnnormcone}). In an arbitrary regular semigroup \cite{cross}, Nambooripad used a pair of \emph{cross-connected cones} to represent a typical element. 
	
	Notice that in \cite{cross}, Nambooripad also briefly outlined the construction of left reductive regular semigroups\footnote{Nambooripad had reversed the convention and called these semigroups right reductive in \cite{cross}. We shall follow Clifford and Preston \cite[Section 1.3]{clif} wherein these semigroups are named left reductive. This will coincide with the convention in $\gl$-unipotent semigroups later (see Section \ref{seccxnlus}).} 
	inside his framework of cross-connections.  However, his construction involved two normal categories and their duals, certain $\mathbf{Set}$-valued functors, and the rather sophisticated definition of cross-connection relating all these categories (see \cite[Section 5]{tx} for a concrete construction of a left reductive regular semigroup using cross-connections). In contrast, taking advantage of the less complicated structure on hand, our construction uses just one normal category and bypasses most of the complicated tools of \cite{cross}, including the cross-connections. In this way, our construction drastically reduces the entry threshold of the cross-connection approach to the structure of semigroups.
	
	The paper is divided into seven sections. After this introduction, in Section \ref{secprel}, we briefly recall the essential preliminary notions regarding semigroups and categories. We also discuss the initial layer of our construction including the notion of  normal category  from Nambooripad's treatise \cite{cross}. In Section \ref{seccxnlrs}, we bifurcate ourselves from \cite{cross} and introduce connected categories. We describe the structure of left reductive regular semigroups using connected categories in Section \ref{ssecslrs} and prove a category equivalence in Section \ref{sseccateqlrs}. We specialise the discussion to the category of $\gl$-unipotent semigroups in Section \ref{seccxnlus}. This culminates in a category equivalence of $\gl$-unipotent (and $\gr$-unipotent) semigroups and \emph{supported}  categories. In Section \ref{seccxnrrb}, we use our construction to describe a category adjunction between the category of supported categories and the category of right (and left) regular bands. In Section \ref{secinvs}, we discuss what is arguably the most important class of semigroups: the class of inverse semigroups. Here, we introduce \emph{self-supported} categories which capture the isomorphism between the left and the right ideal structures in inverse semigroups. This leads to a category equivalence just as in \cite{locinverse} but this is much simpler and sheds more light on the symmetry of these semigroups. Finally, in Section \ref{secrm}, we discuss regular monoids and totally left reductive semigroups. One of the significant results in this section is Theorem  \ref{thmnccat} which identifies the category of semigroups corresponding to the category of normal categories (also see Corollary \ref{cortlclr}). We describe regular monoids using bounded above normal categories and provide concrete descriptions of connected categories for semigroups in some important classes such as transformation semigroups, linear transformation semigroups and symmetric inverse monoids. These new descriptions subsume (and improve) the discussions in \cite{tx,tlx} and also illustrate the precision of our construction.

	The following table lists the various categories of left reductive regular semigroups considered in the paper and their corresponding categories of connected categories:
	
	\begin{center}
		\vspace{.5cm}
		\begin{tabular}{lcl}
			\hline
			Semigroups&& Categories\\
			\hline\\
			$\mathbf{LRS}$- left reductive regular semigroups&&  $\mathbf{CC}$- connected categories\\[3pt]
			$\mathbf{LUS}$- $\gl$-unipotent semigroups&&  $\mathbf{SC}$- supported categories\\[3pt]
			$\mathbf{RRB}$- right regular bands&& --- \\[3pt]
			$\mathbf{RRS}$- right reductive regular semigroups&&  $\mathbf{CC}$- connected categories\\[3pt]
			$\mathbf{RUS}$- $\gr$-unipotent semigroups&&  $\mathbf{SC}$- supported categories\\[3pt]
			$\mathbf{LRB}$- left regular bands&&  --- \\[3pt]
			Inverse semigroups&&  Self-supported categories\\[3pt]
			Totally left reductive semigroups&&  Normal categories\\[3pt]
			Regular monoids && Bounded above normal categories\\[3pt]
			$\tn$- full transformation monoid&&  $\fpn_\Pi$ (or simply $\fpn$) - full powerset category \\[3pt]
			$\sts$- singular transformation semigroup&&  $\pn$- singular powerset category \\[3pt]
			$\tv$- linear transformation monoid&&  $\sv$- full subspace category \\[3pt]
			$\stv$- singular linear transformation semigroup&&  $\psv$- singular subspace category \\[3pt]
			$\ix$- symmetric inverse monoid&&$\x$- partial bijection subsets category\\[3pt]
			\hline
		\end{tabular}
		\vspace{1cm}
	\end{center}

	\section{Preliminaries}\label{secprel}
	
	We assume familiarity with some basic ideas from category theory and semigroup theory. For undefined notions, we refer to \cite{mac,higginscat} for category theory and \cite{clif,howie,grillet} for semigroups and biordered sets. In the sequel, all mappings, morphisms and functors shall be written in the order of their composition, i.e., from left to right. 
	
	\subsection{Semigroups and categories}\label{sssemicat}
	
	All the semigroups we discuss in this paper are regular. In a regular semigroup, each element $a$  admits at least one \emph{inverse}  $x$, meaning that $axa=a$ and $xax=x$. The set of all such inverses of $a$ shall be denoted by $V(a)$. Given a regular semigroup $S$, we can define two quasi-orders $\lel$ and $\ler$ on $S$ as follows. For $a,b\in S$:
	$$a\lel b \iff a\in Sb \:\text{ and }\: a\ler b \iff a\in bS.$$
	Then the Green's relations $\gl$ and $\gr$ are the equivalence relations defined on the semigroup $S$ as follows: 
	$$\gl~:=~\lel \cap (\lel)^{-1}\text{ and }\gr~:=~\ler \cap (\ler)^{-1}.$$ 
	Given an element $a\in S$, we shall denote the $\gl$ and $\gr$ classes of $a$ by $L_a$ and $R_a$, respectively. The natural partial order $\leqslant$ on a regular semigroup $S$ is given by $\leqslant~:=~\lel \cap \ler$ \cite{kss1980}. In the sequel, we shall denote the restrictions of the above mentioned relations to the set $E(S)$ of idempotents of $S$ by the same symbols. Recall that a \emph{down-set} of a poset $(P,\le)$ is a subposet $D$ of $P$ such that $p\le d\in D$ implies $p\in D$. Given an idempotent $e\in E(S)$, we shall denote the \emph{down-set generated by} $e$ as $\omega(e):=\{ f\in E(S): f \leqslant e\}$, where $\leqslant$ is the natural partial order on a semigroup $S$. 
	
	In this paper, we shall deal with two types of categories. The first type are locally small categories whose objects are algebraic structures and morphisms are structure preserving mappings. An example of this type is the category $\mathbf{LRS}$ of left reductive regular semigroups whose objects are left reductive semigroups and morphisms are semigroup morphisms. These categories will be dealt with in a standard way and we will be concerned about adjunctions and equivalences between such categories. The second type are small categories which are treated as algebraic structures by themselves; for example, a connected category as in Definition \ref{dfncncat}. Observe that these are specialised small categories and carry additional structure. So, when comparing such small categories, we shall employ stronger notions such as CC-morphisms (see Definition \ref{dfnccmor}). It is worth mentioning that we shall also consider locally small categories (i.e., of the first type) whose objects are small categories of the second type. For example, the category $\mathbf{CC}$ of connected categories has connected categories as objects and CC-morphisms as morphisms.
	
	Given any locally small category $\mathcal C$, the class of objects of $\mathcal C$ is denoted by $v\mathcal C$, and the set of morphisms by $\mathcal C$ itself. If $c,c' \in v\mathcal{C}$ are two objects in the category $\mathcal{C}$, the set of all morphisms from $c$ to $c'$ in $\mathcal{C}$ is denoted by $\mathcal{C}(c,c')$. 
	
	\subsection{Regular semigroups and normal categories}\label{ssecnc}
	
	Now, we proceed to give a quick introduction regarding the notion of normal categories and how these categories characterise the principal left ideals of a regular semigroup.   
	
	Recall that a morphism in a category is called a \emph{monomorphism} if it is right cancellable and an \emph{epimorphism} if it is left cancellable. A morphism $f\colon c\to c'$ in a category $\mathcal{C}$ is said to be an \emph{isomorphism} if there exists a morphism $g\colon c'\to c$ in $\mathcal{C}$ such that $fg=1_c$ and $gf=1_{c'}$. 
	
	A \emph{preorder} category is a category such that there is at most one morphism between any two given (possibly equal) objects. A strict preorder $\mathcal{P}$ is a preorder in which the only isomorphisms are the identity morphisms. In a strict preorder, the relation $\preceq$ on the class $v\mathcal{P}$ defined by:
	$$p\preceq q \iff \mathcal{P}(p,q) \ne \phi \text{ for } p,q \in v\mathcal{P}$$
	is a partial order. Hence, a \emph{small} strict preorder category $\mathcal{P}$ is equivalent to a poset $(v\mathcal{P},\preceq)$.
	
	\begin{dfn}\label{catsub}
		Let $\mathcal{C}$ be a small category and $\mathcal{P}$ a subcategory of $\mathcal{C}$. Then the pair $(\mathcal{C},\mathcal{P})$ (often denoted simply by $\mathcal{C}$) is said to be a \emph{category with subobjects} if
		\begin{enumerate}
			\item $\mathcal{P}$ is a strict preorder with $v\mathcal{P}=v\mathcal{C}$;
			\item every $f\in \mathcal{P}$ is a monomorphism in $\mathcal{C}$;
			\item if $f,g\in \mathcal{P}$ and if $f=hg$ for some $h\in \mathcal{C}$, then $h\in \mathcal{P}$.
		\end{enumerate}
	\end{dfn}
	
	Before proceeding further, we introduce a concrete example of a category with subobjects. We shall return to this example repeatedly in the sequel, most notably in Section \ref{ssectrans}. This running example is designed to prepare the reader for the layered construction of the connected category developed in the next section. We shall recall ideas from \cite{tx,tx1,tx2} and use them to illustrate a concrete example of a connected category in Section \ref{ssfpn}. We also rectify an error in these papers by assuming the underlying set to be finite rather than arbitrary.
	
	\begin{ex}\label{exfpn0}
		(Full power set category, $\fpn$) \: Let $\ns:=\{1,\dots,n\}$. Then the set of all subsets of $\ns$ forms a small category $\fpn$ with mappings as morphisms, i.e.,
		$$v\fpn:=\{A: A\subseteq \ns\}\:\text{ and }\:\fpn(A,B):=\{f: f \text{ is a mapping from }A \text{ to } B \}.   $$
		Also, given $A\subseteq \ns$, the identity map $1_A$ is the identity morphism at the object $A$ in $\fpn$. Observe that $\fpn$ is a small, full subcategory of the large (in fact, locally small) category $\mathbf{Set}$.  
		
		Now, we can realise $\fpn$ as a category with subobjects $(\fpn,\mathcal{P})$ as follows. Let $v\mathcal{P}:=v\fpn$ and for subsets $A\subseteq B \subseteq \ns$, we define $$\mathcal{P}(A,B):= i(A,B)$$ 
		where $ i(A,B)$ is the set inclusion map from $A$ to $B$. Then  $\mathcal{P}$ is a strict preorder category, or equivalently $(v\mathcal{P},\subseteq)$ is a poset. Observe that monomorphisms in $\fpn$ are precisely injective maps. So, it is routine to verify that the pair  $(\fpn,\mathcal{P})$ satisfies Definition \ref{catsub} and hence forms a category with subobjects.
	\end{ex}

	Motivated in part by the previous example, in an arbitrary category $(\mathcal{C} ,\mathcal{P})$ with subobjects, the morphisms in $ \mathcal{P}$ will be called \emph{inclusions}. Since $\mathcal{P}$ is a strict preorder, it follows that there is a unique inclusion between two given objects. If $c \to c'$ is an inclusion, we have $c\preceq c'$, and we denote this inclusion by $j(c,c')$. An inclusion $j(c,c')$ \emph{splits} if there exists $q\colon c'\to c \in \mathcal{C}$ such that $j(c,c')q =1_{c}$, and a morphism $q$ satisfying such an equality is called a \emph{retraction}.
	
	\begin{dfn}\label{normfact}
		Let $\mathcal{C}$ be a category with subobjects. A morphism $f$ in $\mathcal{C}$ is said to have a \emph{normal factorisation} if $f=quj$, where $q$ is a retraction, $u$ is an isomorphism and $j$ is an inclusion, respectively, in $\mathcal{C}$.
	\end{dfn}
	
	In general, a normal factorisation of a morphism in a category with subobjects need not be unique. However as in \cite[Proposition II.5]{cross}, we can see that given two normal factorisations $f=quj=q'u'j'$ of a morphism $f$, we always have $qu=q'u'$ and $j=j'$. So for a normal factorisation $f=quj$, the  epimorphism $qu$ does not depend on the factorisation. Hence any morphism $f$ in a category with subobjects has a \emph{unique} factorisation of the form $f=pj$, where $p=qu$ is an epimorphism and $j$ is an inclusion. Such a factorisation is called as a \emph{canonical factorisation} of the morphism $f$. 
	
	The morphism $p$ from the canonical factorisation is known as the  \emph{epimorphic component} of the morphism $f$ and is denoted in the sequel by $f^\circ$. Similarly, the morphism $j$ is known as the inclusion of $f$ and denoted by $j_f$. The codomain of $f^\circ$ is called the \emph{image} of $f$ and shall be denoted as  $\im f$. Likewise, the codomain of the retraction $q$ is called the \emph{coimage} of $f$ and denoted by $\coim f$. Observe that for a given morphism $f$, although $\im f$ is unique, $\coim f$ need not be uniquely defined. We collect the following results as a lemma which will be quite useful in the sequel for manipulating expressions involving morphisms.
	
	\begin{lem}[{\cite[Corollary II.4, Proposition II.5 and II.7]{cross}}]\label{lemepi}
		Let $\mathcal{C}$ be a category  where all inclusions split and every morphism has a normal factorisation. Then the following are true.
		\begin{enumerate}
			\item Every morphism $f$ in $\mathcal{C}$ has a unique epimorphic component $f^\circ$ i.e., $f^\circ$ is independent of the chosen normal factorisation of $f$.
			\item If $p$ is an epimorphism, then the epimorphic component $p^\circ=p$.
			\item If $f$ and $g$ are composable morphisms such that the inclusion of $f$ is $j_f$, then 
			$$(fg)^\circ=f^\circ (j_f g)^\circ.$$
			\item The inclusion of an epimorphism $p\in\mathcal{C}$ is the identity morphism, and so every normal factorisation of $p$ is of the form $p=qu$, where $q$ is a retraction and  $u$ is an isomorphism.
			\item Dually, the retraction of a monomorphism $m\in\mathcal{C}$ is the identity morphism, and so every normal factorisation of $m$ is of the form $m=uj$, where  $u$ is an isomorphism and $j$ is an inclusion. In particular, the epimorphic component of an inclusion is the identity morphism.
		\end{enumerate}	
	\end{lem}
	
	\begin{ex}\label{exfpn1}
		Let $\fpn$ be the full power set category  of Example \ref{exfpn0}. Given an inclusion map $i(A,B)$ in $\fpn$, we can always find a retraction map $q\colon B\to A$ in the category $\fpn$ such that $i(A,B)\:q = 1_A$. Observe that this retraction $q$ need not be unique, in general. Any mapping $f$ in $\fpn$ from $A$ to $B$ has a uniquely defined image, and so the image of the morphism $f$ in the category $\fpn$ is the image $Af$ of the mapping. 
		
		Further, given a mapping $f$ from $A$ to $B$, let $B':= Af$ be the image of the mapping $f$ so that $j:=i(B',B)$ is an inclusion map. Now, the map $f$ determines a partition of $A$ given by:
		\begin{equation}\label{eqnpi}
			\pi_f:=\{ (x,y)\in A\times A ~:~xf=yf \}.
		\end{equation}
		Let $A'$ be a cross-section of the partition $\pi_f$, and given an arbitrary $a\in A'$, let $[a]$ be the equivalence class of $\pi_f$ in the set $A$ containing $a$. Define $q\colon A\to A'$ as the surjection given by  $q\colon [a]\mapsto a $ and then we have $i(A',A)q=1_{A'}$. Also, $u:=f_{|A'}$ will be a bijection from $A'$ to $B'$. Hence we have $f= quj$ as illustrated in the diagram below where $q$ is a retraction, $u$ is an isomorphism and $j$ is an inclusion in the category $\fpn$. 
		\begin{equation*}\label{normalfact}
			\xymatrixcolsep{4pc}\xymatrixrowsep{4pc}\xymatrix
			{
				A' \ar@{.>}@<-3pt>[r]_{\subseteq} \ar@<-3pt>[d]_{u}   & A \ar@<-3pt>[l]_{q}\ar[d]^{f} \\       
				B' \ar@{.>}@<-3pt>[u]_{u^{-1}} \ar[r]^{\subseteq}_{j} & B 
			}
		\end{equation*}
		Summarising, every inclusion in the category $\fpn$ splits and  any morphism $f$ in $\fpn$ has a normal factorisation of the form $quj$. Observe that although the retraction $q$ is not unique (as it depends on the choice of the cross-section $A'$), the epimorphic component $f^\circ=qu\colon A \to B'$ is always unique.
	\end{ex}
	
    The next definition introduces the basic building blocks of the semigroup constructed from a category with subobjects. 
	\begin{dfn}\label{dfnnormcone}
		Let $\mathcal{C}$ be a category with subobjects and $z\in v\mathcal{C}$. A mapping $\gamma$ from $v\mathcal{C}$ to $\mathcal{C}$ defined by $\gamma\colon c\mapsto\gamma(c)\in\mathcal{C}(c,z)$  for each $c\in v\mathcal{C}$, is said to be a \emph{cone}\footnote{cones were called as \emph{normal cones} in \cite{cross}.} with vertex $z$ if:
		\begin{enumerate}
			\item whenever $a\preceq b$, we have $j(a,b) \gamma(b)=\gamma(a)$;
			\item there exists at least one $c\in v\mathcal{C}$ such that $\gamma(c)\colon c\to z$ is an isomorphism.
		\end{enumerate}
	\end{dfn}
	Given a cone $\gamma$, we denote by $z_\gamma$ the vertex of $\gamma$ and the morphism $\gamma(c)$ is called the \emph{component} of the cone $\gamma$ at the object $c$. The figure \ref{figcone} illustrates a typical cone $\gamma$ with vertex $z$ in a category $\mathcal{C}$.
	
	\begin{figure}[ht]
		\begin{center}
			\begin{tikzpicture}[scale=.5]
				\tikzset{shorten >= 10pt}
				\foreach \x in {-5,-1,1,3,7} {\draw[->] (\x,5) to (1,-2);} 
				
				\node () at (1,-2) {\footnotesize $z$};
				\node () at (-5.3,5.5) {\footnotesize $a$};
				\node () at (-1.1,5.5) {\footnotesize $b$};
				\node () at (1,5.5) {\footnotesize $c$};
				\node () at (3.1,5.5) {\footnotesize $z$};
				\node () at (7.3,5.5) {\footnotesize $x$};
				
				\node () at (-4,3) {\footnotesize $\gamma(a)$};
				\node () at (-1,3) {\footnotesize $\gamma(b)$};
				\node () at (1.2,3) {\footnotesize $\gamma(c)$};
				\node () at (3.1,3) {\footnotesize $\gamma(z)$};
				\node () at (6,3) {\footnotesize $\gamma(x)$};

				\draw[->] [densely dotted] (-4.8,5) to node[below] {\scriptsize $j(a,b)$} (-0.7,5) ;
				\draw [loosely dotted] (3.5,5.5) to   (7.2,5.5) ;
				
			\end{tikzpicture}
		\end{center}
		\caption{A cone $\gamma$ with vertex $z_\gamma=z$}
		\label{figcone}
		
	\end{figure}
	
	\begin{dfn}\cite[Section III.1.3]{cross}\label{dfnnormc}
		A category $\mathcal{C}$ is said to be a \emph{normal category} if:
		\begin{enumerate}
			\item [(NC 1)] $\mathcal{C}$ is a category with subobjects;
			\item [(NC 2)] every inclusion in $\mathcal{C}$ splits;
			\item [(NC 3)] every morphism in $\mathcal{C}$ admits a normal factorisation;
			\item [(NC 4)] for each $c\in v\mathcal{C}$ there exists a cone $\mu$ with vertex $c$ such that $\mu(c)=1_c$.
		\end{enumerate}
	\end{dfn}
	
	As previously alluded in Section \ref{sssemicat}, a normal category is a small category with extra structure, i.e., of the second type, and hence the functors between such categories need to respect the extra structure. So for two normal categories, we consider inclusion preserving functors between them. Given a functor $F$, we shall denote its object map by $vF$ and the  morphism map by $F$ itself. Recall that a functor $F$ is said to be \emph{faithful} and \emph{full} if the morphism map $F$ is injective and surjective, respectively, and a functor with both properties is called \emph{fully-faithful}. Two normal categories are said to be \emph{isomorphic} if we have an {inclusion preserving} functor $F$ which is fully-faithful such that the object map $vF$ is a bijective order isomorphism.
	
		\begin{ex}\label{exfpn2}
		The full power set category $\fpn$ (see Examples \ref{exfpn0} and \ref{exfpn1}) is a normal category. We have already seen that $\fpn$  is a category with subobjects such that every inclusion splits and any morphism in $\fpn$ has a normal factorisation. Now, for any object $A\subseteq\ns$ let $\alpha\colon \ns\to A$  be a mapping such that $\alpha_{|A}=1_A$. Then for each subset $S\subseteq \ns$, letting $\epsilon(S):=\alpha_{|S}$, we see that  $\epsilon$ is a cone in $\fpn$ with vertex $A$ such that $\epsilon(A)=\alpha_{|A}=1_A$.  Hence $\fpn$ is normal.
	\end{ex}
	
	Let $\mathcal{C}$ be a normal category and let $\gamma$ be a cone in $\mathcal{C}$. If $f\in \mathcal{C}(z_\gamma,z_f)$ is an epimorphism with $\im f=z_f$, then as in \cite[Lemma I.1]{cross}, we can see that the map
	\begin{equation}\label{eqnbin0}
		\gamma\ast f\colon c \mapsto \gamma(c)f \text{ for all } c\in v\mathcal{C}	
	\end{equation} 
	is a cone such that the vertex $z_{\gamma\ast f}=z_f$. Recall that composition of two epimorphisms is again an epimorphism. So in the sequel, given a cone $\gamma$ and two composable epimorphisms $f$ and $g$, we shall write $\gamma\ast f\ast g$ to denote $\gamma\ast (fg)=(\gamma\ast f)\ast g$. 
	
	Now, given two cones $\gamma$ and $\delta$,
	\begin{equation}\label{eqnbin}
		\gamma \cdot\delta=\gamma \ast(\delta(z_{\gamma}))^\circ
	\end{equation}
	where $(\delta(z_{\gamma}))^\circ$ is the epimorphic component of the morphism $\delta(z_{\gamma})$, defines a binary composition on the set of all cones in $\mathcal{C}$. This binary composition on the set of cones is illustrated in Figure \ref{figbin} wherein the components of the composed cone $\gamma\cdot\delta$ are drawn in \textcolor{red}{red} colour. So, for instance, the component of the cone $\gamma\cdot\delta$ at the object $a$ is the morphism $\gamma(a)(\delta(z_{\gamma}))^\circ$. Observe that the vertices satisfy $z_{\gamma\cdot\delta} \preceq z_{\delta}$ but the inclusion $j(z_{\gamma\cdot\delta},z_{\delta})$ need not always be an identity morphism, i.e.,  $z_{\gamma \cdot \delta}\neq z_{\delta}$. In the sequel, we shall often denote the binary composition of cones by juxtaposition.
	\begin{figure}[ht]
		\begin{center}
			\begin{tikzpicture}[scale=.6]
				\tikzset{shorten >= 5pt}
				\foreach \x in {-5,-1,1,3,5,7} {\draw[red,thick,->] (\x,5) to (1,-2);} 
				
				\node () at (1,-2) {\footnotesize $z_\gamma$};
				\node () at (-5.3,5.5) {\footnotesize $a$};
				\node () at (-1.1,5.5) {\footnotesize $b$};
				\node () at (1,5.5) {\footnotesize $z_\gamma$};
				\node () at (3.1,5.5) {\footnotesize $z_{\gamma\cdot\delta}$};
				\node () at (5.2,5.5) {\footnotesize $z_{\delta}$};
				\node () at (7.3,5.5) {\footnotesize $x$};
				\node () at (-3,4.7) {\tiny $j(a,b)$};
				\node () at (3.9,4.7) {\tiny $j(z_{\gamma\cdot\delta},z_{\delta})$};
				\draw [loosely dotted] (5.6,5.5) to   (7.2,5.5) ;
				\draw[->] [densely dotted] (-5,5) to  (-0.7,5) ;
				\draw[->] [densely dotted] (3,5) to  (5.35,5) ;
				
				\node () at (-4,3) {\footnotesize $\gamma(a)$};
				\node () at (-1,3) {\footnotesize $\gamma(b)$};
				\node () at (0.5,3) {\footnotesize $\gamma(z_\gamma)$};
				\node () at (2.2,3) {\footnotesize $\gamma(z_{\gamma\cdot\delta})$};
				\node () at (4.1,3) {\footnotesize $\gamma(z_\delta)$};
				\node () at (6,3) {\footnotesize $\gamma(x)$};

				\node () at (1,-9) {\footnotesize $z_\delta$};
				\node () at (1,-7) {\footnotesize $z_{\gamma\cdot\delta}$};
				\node () at (-5.3,-2) {\footnotesize $a$};
				\node () at (-1.1,-2) {\footnotesize $b$};
				\node () at (3.1,-2) {\footnotesize $z_{\gamma \cdot \delta}$};
				\node () at (5.2,-2) {\footnotesize $z_{\delta}$};
				\node () at (7.3,-2) {\footnotesize $x$};
				\draw[red,thick,->] (1,-2.3) to (1,-7);
				\draw[->] [densely dotted] (1,-7.3) to  (1,-9) ;
				\node () at (0.9,-7.6) {\tiny $j(z_{\gamma\cdot\delta},z_{\delta})$};
				\node () at (0.75,-4) {\footnotesize $(\delta(z_{\gamma}))^\circ$};
				\node () at (6.2,-4) {\footnotesize $\delta(x)$};
				\draw [loosely dotted] (5.6,-2) to   (7.2,-2) ;
				\foreach \x in {-5,-1,1,3,5,7} {\draw[->][dotted] (\x,-2.3) to (1,-9);} 
				\draw[->] [densely dotted] (-5,-2.3) to  (-0.7,-2.3) ;
				\draw[->] [densely dotted] (3,-2.3) to  (5.35,-2.3) ;

			\end{tikzpicture}
		\end{center}
		\caption{Binary composition of cones $\gamma$ (on top) and $\delta$ (below) in a normal category $\mathcal{C}$.}
		\label{figbin}
		
	\end{figure}

	\begin{lem}[{\cite[Theorem I.2]{cross}}]\label{lemrs}
		Let $\mathcal{C}$ be a normal category. A cone $\mu$ in $\mathcal{C}$ is an idempotent if and only if $\mu(z_\mu)=1_{z_\mu}$. The set $ \snc $ of all cones forms a regular semigroup  under the binary composition defined in (\ref{eqnbin}). Given a cone $\gamma \in \snc$ with vertex $z_\gamma$, an inverse cone $\chi$ in $\snc$ is given by $\chi:=\mu(z_\gamma)\ast (\gamma(d))^{-1}$, where $\mu(z_\gamma)$ is an idempotent cone in $\snc$  with vertex $z_\gamma$ and $d\in v\mathcal{C}$ is an object such that $\gamma(d)$ is an isomorphism. 
	\end{lem}
	
	The next two lemmas follow from the discussion in \cite[Section III.2]{cross}. 
	\begin{lem}\label{lemgc}
		Let $\gamma,\delta$ be cones in the regular semigroup $ \snc $. Then the quasi-orders in $\snc$ are characterised as follows. 
		\begin{enumerate}
			\item $\gamma \lel\delta \text{ if and only if }z_{\gamma}\preceq z_{\delta}$, and so $\gamma \gl\delta \text{ if and only if }z_{\gamma}= z_{\delta}$. 
			\item $\gamma \ler \delta \text{ if and only if the component }\gamma(z_{\delta})\text{ is an epimorphism such that }\gamma=\delta \ast\gamma(z_{\delta })$. We have $\gamma\: \gr \:\delta$  if and only if   $\gamma=\delta\ast h$ for a unique isomorphism $h$; in that case $h:=\gamma(z_{\delta})$.
		\end{enumerate} 
	\end{lem} 
	\begin{lem}\label{lemorder}
		Let $\nu,\mu$ be idempotent cones in the semigroup $ \snc $ and let $\leqslant$ be the natural partial order on the set of idempotents of $ \snc $. Then $\nu \leqslant \mu$ if and only if  $\nu(z_{\mu })$ is a retraction such that $\nu=\mu \ast\nu(z_{\mu }).$
	\end{lem}

	Now, we briefly describe how normal categories come from regular semigroups. Given a regular semigroup $S$, we define the category $\los$ of the principal left ideals, called the \emph{left category}, by
\begin{equation}\label{los}
	v\los = \{ Se : e \in E(S) \},
\end{equation}
	and the set of all morphisms from the object $Se$ to the object $Sf$ is the set
	\begin{equation}\label{los1}
	\los(Se,Sf) = \{ r(e,u,f) : u\in eSf \},
	\end{equation}
	where $r(e,u,f)\colon x \mapsto xu $, for each $x\in Se$.

	Given any  morphisms $r(e,u,f)$ and $r(g,v,h)$, they are \emph{equal} if and only if $e\gl g$, $f\gl h$ and $v=gu$ (or $u=ev$ ) ; and they are \emph{composable} if $Sf=Sg$ (i.e., if $f\gl g$) in which case
	$$ r(e,u,f)\: r(g,v,h) :=  r(e,uv,h).$$
	
	Observe that $\los$ has a particular subcategory $\mathcal{P}_\mathbb{L}$ given by $v\mathcal{P}_\mathbb{L}=v\los$ and there exists in $\mathcal{P}_\mathbb{L}$  a morphism from $Se$ to $Sf$ if and only if $Se\subseteq Sf$, this morphism being exactly  $ r(e,e,f)$. The morphisms of  $\mathcal{P}_\mathbb{L}$ correspond to the inclusions of principal ideals. Consequently $\mathcal{P}_\mathbb{L}$ is a strict preorder and $(\los,\mathcal{P}_\mathbb{L})$ is a category with subobjects. Given an inclusion $ r(e,e,f)\in \mathcal{P}_\mathbb{L}$, it has a right inverse $ r(f,fe,e)$ in $\los$; thus every inclusion in the category $\los$ splits. 
	
	Let $ r(e,u,f)$ be an arbitrary morphism in $\los$. Then as shown in \cite[Corollary III.14]{cross}, we can see that there exist $h\in E(L_u)$ and $g\in E(R_u)\cap \omega(e)$ such that
	$$ r(e,u,f)=  r(e,g,g) r(g,u,h) r(h,h,f),$$
	where $ r(e,g,g)$ is a retraction, $ r(g,u,h)$ is an isomorphism and $ r(h,h,f)$ is an inclusion. This is a normal factorisation of the morphism $ r(e,u,f)$. Observe that the image of the morphism $r(e,u,f)$ is uniquely determined and it is the principal left ideal $Sh=Su$, but there is a choice for the coimage $Sg$. This shows that in an arbitrary regular semigroup, although the image of a morphism is unique, the coimage need not be unique. 
	
	Further, if $a$ is an arbitrary element of a regular semigroup $S$, then for each $Se\in v\los$,  the mapping $ r^a\colon v\los\to \los$ defined by
	\begin{equation}\label{eqnprinc}
		r^a(Se) :=  r(e,ea,f),\text{ where } f\in E(L_a)
	\end{equation}
	is a cone with vertex $Sf$, usually referred to as a {\em principal cone} in the category $\los$. Observe that, for an idempotent $e\in E(S)$, we have a principal cone $ r^e$ with vertex $Se$ such that $ r^e(Se)=  r(e,e,e)=1_{Se}$. Summarising the above discussion, it can be easily verified that $\los$ is a normal category \cite[Theorem III.16]{cross}. 
	
	Conversely, given an abstractly defined normal category $\mathcal{C}$, we obtain a regular semigroup $ \snc $ of cones in $\mathcal{C}$. Then the left category $\mathbb{L}( \snc )$ of the semigroup $\snc$ is isomorphic to $\mathcal{C}$ \cite[Theorem III.19]{cross}. We shall later give an independent proof of this fact as a consequence of our results (see Proposition \ref{procatiso} and Corollary \ref{cortlclr}). 
	
	\begin{thm}[{\cite[Corollary III.20]{cross}}] \label{thmnls}
	A small category $\mathcal{C}$ is normal if and only if $\mathcal{C}$ is isomorphic to a category $\los$, for some regular semigroup $S$.
	\end{thm}

	It is worth mentioning here that although $\mathbb{L}( \snc ) \cong \mathcal{C}$ for a normal category $\mathcal{C}$, we \textit{do not} have $\sls$ isomorphic to $S$ for an arbitrary regular semigroup $S$. This relationship in general is more subtle, as described in Theorem \ref{thmlslr} below. So, every normal category comes from a regular semigroup although not every regular semigroup can be constructed from a normal category.
	
	Now, we shift our focus to the subclass of regular semigroups which can indeed be constructed using just one normal category. Recall that a regular semigroup $S$ is said to be \emph{left reductive} if the regular representation $\rho$ is injective. In the category $\los$, the cones are direct abstractions of the regular representation of a semigroup.  In fact, it can be shown that the semigroup $\{ r^a: a\in S\}$ of all principal cones in $\los$ is isomorphic to the image $S\rho$  of the regular representation of the semigroup $S$. Roughly speaking, the left `part' of the regular semigroup $S$ is captured by the normal category $\los$.
	
	\begin{thm}[{\cite[Theorem III.16]{cross}}]\label{thmlslr}
		Let $S$ be a regular semigroup. There is a homomorphism $\bar{\rho}\colon S \to \sls$ given by $a\mapsto r^a$. Further, $S$ is isomorphic to a subsemigroup of $\sls$ (via the map $\bar{\rho}$) if and only if $S$ is left reductive.
	\end{thm}

	Dually, we define the normal category $\ros$ of principal right ideals of a regular semigroup $S$ by:
	\begin{equation}\label{eqnrs}
		v\ros = \{ eS : e \in E(S) \} \: \text{ and }  \: \ros(eS,fS) = \{  l(e,u,f) : u\in fSe \}
	\end{equation}
	where a morphism from $eS$ to $fS$ is the mapping $ l(e,u,f)\colon x \mapsto ux$ for each $x\in eS$.

	\section{Left reductive regular semigroups} \label{seccxnlrs}
	
	We proceed to give a construction for a left reductive regular semigroup  as a subsemigroup of the semigroup $ \snc $ of cones of a normal category $\mathcal{C}$. \emph{This is where we bifurcate from Nambooripad's construction. }
	
	\subsection{Connected categories}\label{sseccc}
	
	First, recall that given any regular semigroup $S$, the set $S/\gr$ forms a  poset  as follows: 
	\begin{equation}\label{eqnsr}
		R_e \sqsubseteq R_f \iff eS \subseteq fS \iff e\ler f.
	\end{equation}
	In fact, the poset $ (S /\gr,\sqsubseteq) $ has been characterised by Grillet as a \emph{regular poset}\footnote{Since the definition involves several new notions and as we do not explicitly use any of the properties of regular posets, we omit the formal definition.} in \cite{grilcross}. Now given a normal category $\mathcal{C}$, since $ \snc $ is a regular semigroup (Lemma \ref{lemrs}),  the poset $ (\snc /\gr,\sqsubseteq) $ is a regular poset.    We are now ready to give the most important definition of this paper.
	\begin{dfn}\label{dfncncat}  
		Let $\mathcal{C}$ be a normal category and let $\mathfrak{D}$ be a down-set of the poset $ \snc /\gr $. Then $\mathcal{C}$ is said to be \emph{connected} by $\mathfrak{D}$ if for every $c \in v\mathcal{C}$, there is some $\mathfrak{d} \in \mathfrak{D}$ such that $\mathfrak{d}$ contains some idempotent cone with vertex $c$. We denote such a category by $\cod$ and say that the regular poset $\mathfrak{D}$ \emph{connects} the normal category $\mathcal{C}$. 
	\end{dfn}
	
	Given such a normal category $\cod$, we define 
	
	$$ \sncd  := \{ \: \gamma\in  \snc : R_{\gamma} \in \mathfrak{D} \: \}.$$
	
	Observe that each idempotent cone $\epsilon$ in the set $ \sncd $ may be uniquely represented as $\epsilon(c,\mathfrak{d})$ such that the vertex $z_{\epsilon}=c$ and $R_{\epsilon}=\mathfrak{d}$ where $c\in v\mathcal{C}$ and $\mathfrak{d}\in \mathfrak{D}$. In this case, we shall say that the object $c$ \emph{ is connected by } $\mathfrak{d}$. Hence, we have:
	
	\begin{equation}\label{eqnesncd}
		E(\sncd) =\{\: \epsilon(c,\mathfrak{d}) : c \text{ is connected by } \mathfrak{d} \: \}.
	\end{equation}
	
	\begin{rmk}\label{rmkcc}
		The definition of the set $\sncd$ involves picking certain $\gr$-classes from the semigroup $\snc$, using the down-set $\mathfrak{D}$. We could have equivalently defined connected categories by letting $\mathfrak{D}$ be a down-set \emph{isomorphic} to a down-set of the poset $\snc/\gr$. Admittedly, this would complicate the discussion substantially and so we avoid it at this stage. However, we shall indeed use this identification in Section \ref{ssfpn} below and later in Section \ref{secrm}, where we discuss concrete cases, as such an identification will lead to simpler descriptions of the connecting posets $\mathfrak{D}$.
		
	\end{rmk}
	\begin{rmk}\label{rmkid}
		By Definition \ref{dfncncat}, every $c\in v\mathcal{C}$ is connected by at least one $\mathfrak{d}\in \mathfrak{D}$ and conversely each $\mathfrak{d}$ connects at least one $c$. As we shall see later, this is a reflection of the fact that every $\gl$ and $\gr$ class of a regular semigroup contains at least one idempotent. Notice that, in general, one object $c\in\mathcal{C}$ may be connected by multiple $\mathfrak{d}\in\mathfrak{D}$, and also different objects in $\mathcal{C}$ may be connected to the same $\mathfrak{d}\in\mathfrak{D}$.  
	\end{rmk}
	
	\subsection{Full powerset category as a connected category}\label{ssfpn}
	Before proceeding further, we return to our running example of full powerset category $\fpn$ and illustrate a concrete example of a connected category. The discussion in Example \ref{exfpn2} shows  that $\fpn$ is a normal category and a cone in the category $\fpn$ is determined by a mapping from  $\ns$ to itself. We shall see below that in fact, this relationship is much stronger. The Lemma \ref{lemfpn2} below may be deduced from \cite[Theorem 3.1]{tx}. Nevertheless, we include a fresh proof to familiarize the reader with the example, thereby facilitating a clearer understanding of our construction.
	
	Recall that the semigroup of all mappings from a finite set $\ns$ to itself, under composition of  mappings is known as the \emph{full transformation monoid} $\tn$.
		\begin{lem}\label{lemfpn2}
			The semigroup $\sfpn$ of cones in the category $\fpn$ is isomorphic to the full transformation monoid $\tn$. 
		\end{lem}
		\begin{proof}
			First, observe that the normal category $\fpn$ has a largest object, namely $\ns$. So given a cone $\gamma$ in the category $\fpn$ with vertex $Z$, we may define
			\begin{equation}\label{eqphi}
				\phi\colon \sfpn \to \tn\text{ given by }\gamma\mapsto\gamma(\ns)i(Z,\ns),
			\end{equation}  
			where the mapping $\gamma(\ns)i(Z,\ns)\colon \ns \to \ns $ is an element of the semigroup $\tn$, and so $\phi$ is well-defined. We proceed to prove that
			$\phi$ is an isomorphism.
			For a cone $\gamma$, by Definition \ref{dfnnormcone} (2) there is some $C\subseteq \ns$ such that $\gamma(C)\colon C\to Z$ is a bijection. However, since $C\subseteq \ns$, the component $\gamma(\ns)$ is always a surjection and so by Lemma \ref{lemepi} (2), we have $(\gamma(\ns))^\circ= \gamma(\ns)$. Hence the expression $\gamma(\ns)i(Z,\ns)$ is in fact the unique canonical factorisation of that mapping. Also notice that by Definition \ref{dfnnormcone} (1), for each $A\subseteq \mathbf{n}$, we have $\gamma(A)=i(A,\ns)\gamma(\ns)$. 
			
			Now, to verify that $\phi$ is a homomorphism, let $\gamma_1,\gamma_2$ be  cones in the category $\fpn$ with vertices $Z_1$ and $Z_2$, respectively. If we denote the vertex $z_{\gamma_1\gamma_2}$ of the cone $\gamma_1\gamma_2$ by $Z$, since $z_{\gamma_1\gamma_2} \subseteq z_{\gamma_2}$ we see that $Z\subseteq Z_2$. Then using equations (\ref{eqnbin0}) and  (\ref{eqnbin}), and Lemma \ref{lemepi} (3), we have
			$$(\gamma_1\:\gamma_2)\phi=(\gamma_1\ast (\gamma_2( Z_1))^\circ)\phi=\gamma_1(\ns)\: (\gamma_2( Z_1))^\circ i(Z,\ns)= \gamma_1(\ns)\: (\gamma_2( Z_1))^\circ i(Z, Z_2)i( Z_2,\ns)=\gamma_1(\ns)\: \gamma_2( Z_1) i( Z_2,\ns).$$
			The last equality of the deduction above is a consequence of the canonical factorisation of the morphism $\gamma_2( Z_1)$ as  $(\gamma_2( Z_1))^\circ i(Z, Z_2)$.
			Also by the definition of $\phi$, since 
			$$\gamma_1\phi\:\gamma_2\phi=\gamma_1(\ns)i( Z_1,\ns)\:\gamma_2(\ns)i( Z_2,\ns)=\gamma_1(\ns)\:\gamma_2( Z_1)i( Z_2,\ns),$$ we see that $\phi$ is a homomorphism.
			
			To show that $\phi$ is injective, let $\gamma_1(\ns)i(Z_1,\ns)=\gamma_2(\ns)i(Z_2,\ns)$. Since these expressions are exactly the unique canonical factorisations of these mappings, we have $\gamma_1(\ns)=\gamma_2(\ns)$. Now $\fpn$ has a largest object $\ns$ and so every cone $\gamma$ is uniquely determined by its component $\gamma(\ns)$, whence $\gamma_1=\gamma_2$.  
			
			Finally, to verify that $\phi$ is a surjection, given an arbitrary mapping $\alpha$ in the monoid $\tn$, for each $S\subseteq \ns$, the map $\gamma(S):=\alpha_{|S}$ is a cone with vertex $\ns \alpha$ such that $\gamma\phi=\gamma(\ns)i(\ns \alpha, \ns)=\alpha$. We conclude that $\phi$ is a semigroup isomorphism.
		\end{proof}
		To realise the category $\fpn$ as a connected category, we need the characterisation of the Green's $\gr$-relation on the regular semigroup $\sfpn$. By the above lemma, the poset $\sfpn/\gr$ is order isomorphic to $\tn/\gr$. So, we recall the following well known results regarding the Green's relations in the monoid $\tn$.
		\begin{lem}[{\cite[Section 2.2]{clif}}]\label{lemgrtn}
			Let $\alpha,\beta$ be arbitrary mappings in $\tn$. 
			\begin{enumerate}
				\item  $\tn\alpha\subseteq \tn\beta$ if and only if $\ns\alpha \subseteq \ns \beta$. Hence $\alpha\gl \beta$ if and only if $\ns \alpha=\ns \beta$. 
				\item  $\alpha\tn \subseteq\beta\tn$ if and only if $\pi_\alpha \supseteq \pi_\beta$. Hence $\alpha\gr \beta$ if and only if $\pi_\alpha=\pi_\beta$. 
			\end{enumerate}	
		\end{lem}
		
		Let $\pi_\alpha$ be the partition induced on $\ns$ by a map $\alpha$  (see (\ref{eqnpi})) and $(\Pi,\supseteq)$ be the poset of all partitions of the set $\ns$. Given an idempotent cone $\epsilon$ in $\sfpn$, using  Lemma \ref{lemfpn2} the mapping $\epsilon(\ns)i(Z,\ns)$ is an idempotent in $\tn$. By a mild abuse of notation, we use  $\pi_\epsilon$ to denote the partition induced by the mapping $\epsilon\phi$  (see (\ref{eqphi})). Define a map $G\colon\sfpn/\gr\to\Pi$ by $R_\epsilon\mapsto \pi_{\epsilon}$. Using Lemmas \ref{lemfpn2} and \ref{lemgrtn} (2), we can routinely verify that $G$ is an order isomorphism. This leads to the following characterisation of the poset $(\sfpn/\gr,\sqsubseteq)$.
		\begin{lem}\label{lemG}
			Let $\gamma_1,\gamma_2$ be cones in the semigroup $\sfpn$. Then $R_{\gamma_1} \sqsubseteq R_{\gamma_2}$ if and only if $\pi_{\gamma_1} \supseteq \pi_{\gamma_2}$. Hence the regular poset $(\sfpn/\gr,\sqsubseteq)$ is order isomorphic to the poset $(\Pi,\supseteq)$ of all partitions of the set $\ns$. 
		\end{lem}
		By the above lemma, we may identify the $\gr$-classes of the semigroup $\sfpn$ by the partitions $\pi\in\Pi$. Summarising, given a finite set $\ns$, the set of subsets of $\ns$ forms a normal category $\fpn$ such that the poset of $\gr$-classes of the semigroup $\sfpn$ is isomorphic to the set $\Pi$ of partitions of $\ns$. This leads us to our first example of a connected category.
		\begin{pro}\label{profpn}
			Given a finite set $\ns$ with full powerset category $\fpn$ and partition poset $\Pi$, the category $\fpn$ is connected by $\Pi$, and so $\fpn_\Pi$ is a connected category.
		\end{pro}
		\begin{proof}
			Given any subset $A\subseteq\ns$, let $\alpha\in \tn$ be such that $\alpha_{|A}=1_A$. Then  $\alpha_{|A}$ is an idempotent in $\tn$. Then for each subset $S\subseteq \ns$, define $\epsilon(S):=\alpha_{|S}$. Now, $\epsilon$ is an idempotent cone in $\sfpn$ such that $\epsilon(\ns) =\alpha$. Further using the isomorphism of Lemma \ref{lemG} and observing that $\epsilon\phi=\alpha$, we have $R_\epsilon\cong\pi_{\epsilon}=\pi_\alpha\in \Pi$. Hence, the subset $A$ is connected by $\pi_\alpha$ and so, the normal category $\fpn$ is connected by the poset $\Pi$.  
		\end{proof}
	
	\begin{rmk}
		In the above example of a connected category  $\cod$, we have $\mathfrak{D}\cong\snc/\gr$. As discussed in Remark \ref{rmkcc}, the relaxation that the ideal $\mathfrak{D}$ is an isomorphic copy of $\snc/\gr$ (rather than $\mathfrak{D}=\snc/\gr$) leads to a concrete characterisation of $\mathfrak{D}$ as the poset $\Pi$. Strictly speaking, with the terminology of Definition \ref{dfncncat}, Proposition \ref{profpn} says that the category $\fpn$ is connected by the poset $\sfpn/\gr$ such that $\sfpn/\gr$ is isomorphic to the poset $\Pi$. 
	\end{rmk}

	\subsection{The connection semigroup $\sncd$}
	
	Having digressed a bit, we now return back to the abstract construction of a left reductive regular semigroup from a connected category $\mathcal{C}$. We shall see that the required semigroup is in fact $\sncd$, which is realised as the subsemigroup of the semigroup $\snc$ of cones in the category $\mathcal{C}$. The following lemma is crucial for the sequel.
	\begin{lem}\label{lemud1}
		Let $\cod$ be a connected category. Then every cone $\gamma$ in the set $ \sncd $ can be expressed as  $\epsilon(c,\mathfrak{d})\ast u$, for some idempotent cone $\epsilon(c,\mathfrak{d})$ and an isomorphism $u$. Conversely, every cone in $\snc$ which can be expressed in this form belongs to  $\sncd$.
	\end{lem}
	\begin{proof}
		First, observe that given a cone $\gamma$ in $ \sncd $, we have $R_\gamma\in \mathfrak{D}$ and let $\mathfrak{d}:=R_\gamma$. Now, since $\mathfrak{D}$ connects the category $\mathcal{C}$, by Remark \ref{rmkid}, there is some $c \in v\mathcal{C}$ such that $\mathfrak{d}$ connects $c$. Let the associated idempotent cone be $\epsilon(c,\mathfrak{d})$.  So we have $\gamma \gr  \epsilon(c,\mathfrak{d})$ in the semigroup $ \snc $. Then by Lemma \ref{lemgc} (2), we get $\gamma=\epsilon(c,\mathfrak{d}) \ast \gamma(c)$ such that $\gamma(c)$ is an isomorphism. 
		
		Conversely, if $\gamma=\epsilon(c,\mathfrak{d})\ast u$ where $u$ is an isomorphism, then by Lemma \ref{lemgc} (2), we obtain $\gamma\gr \epsilon(c,\mathfrak{d})$ and so $R_\gamma=R_{\epsilon(c,\mathfrak{d})}=\mathfrak{d}\in \mathfrak{D}$. Thus $\gamma\in  \sncd $.
	\end{proof}
	
	\begin{rmk}\label{rmkeu}
		Given a cone $\gamma$ with vertex $c$,  the decomposition of $\gamma$ as above is not unique,  in general. Let $\epsilon_1:=\epsilon(c_1,\mathfrak{d})$ and $\epsilon_2:=\epsilon(c_2,\mathfrak{d})$ in be idempotents in $\sncd$ such that $\mathfrak{d}:=R_\gamma$. Then as in the proof of the lemma above, we can see that $\gamma=\epsilon_1\ast \gamma(c_1)=\epsilon_2\ast \gamma(c_2)$ for isomorphisms $\gamma(c_1)\colon	c_1 \to c$ and $\gamma(c_2)\colon c_2 \to c$. Figure \ref{figdeco} illustrates this situation, wherein we consider the `egg-box' diagram of a typical $\gd$-class of the regular semigroup $\snc$. Observe that by Lemma \ref{lemgc}, the $\gl$-classes of $\snc$ are determined by the vertices of the cones.
		\begin{figure}[ht]
			\begin{center}
				\begin{tikzpicture}[scale=.5]
					\foreach \x in {-1,1,3} \foreach \y in {-1,2,5,8,11,14} {\draw (-1,\x)--(14,\x); \draw(\y,-1)--(\y,3);}
					\node () at (0.5,2) {\footnotesize $\epsilon_1$};
					\node () at (12.5,2) {\footnotesize $\gamma$};			
					\node () at (6.5,2) {\footnotesize $\epsilon_2$};
					\node () at (14.5,2) {\footnotesize $\mathfrak{d}$};	
					\node () at (0.5,-1.5) {\footnotesize $c_1$};
					\node () at (6.5,-1.5) {\footnotesize $c_2$};
					\node () at (12.5,-1.5) {\footnotesize $c$};
					
					\draw[->][in=165, out=15]  (6.5,3.5) to node[below] {\scriptsize $\gamma(c_2)$} (12.5,3.5) ;
					\draw[->][in=160, out=20]  (0.5,3.5) to node[above] {\scriptsize $\gamma(c_1)$} (12.5,3.5) ;
					\fill[fill=gray,fill opacity=.25] (-1,3) rectangle (2,1);
					\fill[fill=gray,fill opacity=.25] (2,-1) rectangle (5,1);
					\fill[fill=gray,fill opacity=.25] (5,1) rectangle (8,3);
					\fill[fill=gray,fill opacity=.25] (11,1) rectangle (14,-1);
					\fill[fill=gray,fill opacity=.25] (8,-1) rectangle (11,1);
					
				\end{tikzpicture}
			\end{center}
			\caption{Decomposition of a cone $\gamma$ as an idempotent cone and an isomorphism}
			\label{figdeco}
			
		\end{figure}	 
	\end{rmk}
	
	The following more general variant of Lemma \ref{lemud1} will be useful in the sequel.
	\begin{lem}\label{lemud2}	
		Given a connected category $\cod$, any cone in the set $\sncd$ has a representation of the form $\epsilon\ast p$ for an idempotent cone $\epsilon\in \sncd$ and an epimorphism $p$. Conversely, any cone of this form belongs to $\sncd$. 
	\end{lem}
	\begin{proof}	
		Since every isomorphism is an epimorphism, the first part of the lemma is obvious from Lemma \ref{lemud1}. Conversely, let $\gamma$ be a cone in $\snc$ of the form $\epsilon  \ast p$, where $\epsilon  :=\epsilon(c'  ,\mathfrak{d}'  )$ is an idempotent cone with vertex $c'  $ and $p$ is an epimorphism.  Using Lemma \ref{lemepi} (4), let $p=qu$ be the normal factorisation of the epimorphism $p$ such that $q\colon c' \to c$ and $u\colon c \to z_\gamma$. Then $\gamma=\epsilon  \ast p=\epsilon  \ast q\ast u$.  
		
		Now, since a retraction is, in particular,  an epimorphism, by equation (\ref{eqnbin}) we see that $\epsilon   \ast q$ is a cone in $ \snc $ with vertex $c$. Let $\mu:=\epsilon   \ast q$ and $\mathfrak{d}:=R_{\mu}$ in the semigroup $ \snc $. As $c\preceq c'   $ and there is a unique morphism between $c$ and $c'   $, we have $\epsilon(c)=j(c,c'   )$.  Next observe that 
		$$\mu(c)= \epsilon   (c)\:q=j(c,c'   )q=1_{c} \quad \text{ and }  \quad\mu(c'   )=\epsilon   \ast q(c'   )=\epsilon   (c'   )\: q= 1_{c'   }\: q=q.$$
		Thus $\mu$ is an idempotent, and $\mu=\epsilon   \ast \mu(c'   )$ with $\mu(c'   )$  a retraction. Therefore, by Lemma \ref{lemorder}, we get that $\mu:=\epsilon  \ast q$ is an idempotent cone such that $\mu\leqslant\epsilon  $ in the semigroup $ \snc $. 	In particular,  $\mathfrak{d}=R_\mu \sqsubseteq R_{\epsilon}=\mathfrak{d}'$.  Since $\mathfrak{D}$ is a down-set and $\mathfrak{d}'\in \mathfrak{D}$, we obtain $\mathfrak{d}\in \mathfrak{D}$. So $\mu=\epsilon(c,\mathfrak{d}) \in  \sncd $ and as shown in Figure \ref{figepi}, we have $\gamma=\mu\ast u$, where $\mu$ is an idempotent cone in  $\sncd$ and $u$ is an isomorphism. Hence, using Lemma \ref{lemud1} the cone $\gamma$ belongs to the set $\sncd$.  	
	\end{proof}
	
	\begin{figure}[h]
		\begin{center}
			\begin{tikzpicture}[scale=.3]
				\foreach \y in {9,11,13,15} \foreach \x in {-6,-3,0,3} {\draw (-6,\y)--(3,\y); \draw(\x,9)--(\x,15);}
				\fill[fill=gray,fill opacity=.25] (-6,13) rectangle (-3,15);
				\fill[fill=gray,fill opacity=.25] (-3,13) rectangle (0,11);
				\fill[fill=gray,fill opacity=.25] (0,11) rectangle (3,9);
				
				\node () at (-1.5,12) { $\epsilon$};
				\node () at (-1.4,15.5) {\scriptsize $c'$};
				\draw[->]   (-1.5,8.9) to   (3.3,5.3);
				\node () at (1.6,7.3) {\scriptsize $q$};	
				\node () at (-6.5,12) {\scriptsize $\mathfrak{d}'$};

				\foreach \x in {-1,1,3,5} \foreach \y in {-1,2,5,8,11} {\draw (-1,\x)--(11,\x); \draw(\y,-1)--(\y,5);}
				\node () at (9.5,2) { $\gamma$};
				\node () at (3.5,-1.5) {\scriptsize $c$};
				\node () at (9.5,-1.5) {\scriptsize $z_\gamma$};
				\node () at (3.5,2) { $\mu$};
				\node () at (-1.5,2) {\scriptsize $\mathfrak{d}$};	
				
				\draw[->][in=160, out=20]  (3.4,5.5) to node[above] {\scriptsize $u$} (9.5,5.5) ;
				\fill[fill=gray,fill opacity=.25] (-1,-1) rectangle (2,1);
				\fill[fill=gray,fill opacity=.25] (2,1) rectangle (5,3);
				\fill[fill=gray,fill opacity=.25] (5,3) rectangle (8,5);
				\fill[fill=gray,fill opacity=.25] (8,3) rectangle (11,5);

			\end{tikzpicture}
		\end{center}
		\caption{Decomposition of a cone $\gamma$ as an idempotent cone and an epimorphism.}
		\label{figepi}	
	\end{figure}

	\begin{pro}\label{proud2}
		Let $\cod$ be a connected category. Then $ \sncd $ is a regular semigroup.
	\end{pro}
	\begin{proof}
		First we need to show that $ \sncd $ is a closed subset of $ \snc $. Let $\gamma_1$ and  $\gamma_2$ be two cones in the set $ \sncd $. Then applying Lemma \ref{lemud1} to the cone $\gamma_1$, we have an idempotent cone $\epsilon_1:=\epsilon(c_1,\mathfrak{d}_1)$  and an isomorphism $u_1$ such that $\gamma_1:=\epsilon_1\ast u_1$. Using equation (\ref{eqnbin}), and Lemma \ref{lemepi} (2) and (3), we see that 
		$$\gamma_1 \gamma_2=\gamma_1\ast (\gamma_2(z_{\gamma_1}))^\circ = \epsilon_1\ast u_1\ast (\gamma_2(z_{\gamma_1}))^\circ= \epsilon_1\ast (u_1 (\gamma_2(z_{\gamma_1}))^\circ).$$ 
		As $u_1$ is an isomorphism and $(\gamma_2(z_{\gamma_1}))^\circ$ is an epimorphism, their composition is an epimorphism. Therefore, the cone $\gamma_1\gamma_2$ is of the form $\epsilon\ast p$ such that $\epsilon$ is an idempotent cone in $\sncd$ and $p$ is an epimorphism. Now, applying Lemma \ref{lemud2} we have $\gamma_1\gamma_2\in \sncd$ as shown in Figure \ref{figcomp} and hence $ \sncd $ is a subsemigroup of $ \snc $. 
		\begin{figure}[h]
			\begin{center}
				\begin{tikzpicture}[scale=.4]
					\foreach \y in {9,11,13,15} \foreach \x in {-6,-3,0,3} {\draw (-6,\y)--(3,\y); \draw(\x,9)--(\x,15);}
					\fill[fill=gray,fill opacity=.25] (-6,13) rectangle (-3,15);
					\fill[fill=gray,fill opacity=.25] (-3,13) rectangle (0,11);
					\fill[fill=gray,fill opacity=.25] (0,11) rectangle (3,9);
					
					\node () at (-1.5,12) { $\epsilon_1$};
					\node () at (-1.4,15.5) {\scriptsize $c_1$};
					\node () at (1.6,15.5) {\scriptsize $z_{\gamma_1}$};
					\node () at (-6.5,12) {\scriptsize $\mathfrak{d}_1$};	
					\draw[->]   (-1.5,8.9) to   (3.3,5.3);
					\node () at (1.6,7.3) {\scriptsize $q$};	
					\node () at (1.5,12) { $\gamma_1$};
					\draw[->][in=160, out=20]  (-1.5,16) to node[above] {\scriptsize $u_1$} (1.5,16) ;
					
					\draw[->][in=100, out=340]  (2.5,15.5) to  (9.5,5.7) ;
					\node () at (6,11) {\scriptsize $(\gamma_2(z_{\gamma_1}))^\circ$};
					\draw[->,dashed][in=160, out=20]  (2.0,16) to  (14.5,15) ;
					\node () at (8,16) {\scriptsize $\gamma_2(z_{\gamma_1})$};

					\foreach \x in {-1,1,3,5} \foreach \y in {-1,2,5,8,11} {\draw (-1,\x)--(11,\x); \draw(\y,-1)--(\y,5);}
					\node () at (9.5,2) { $\gamma_1\gamma_2$};
					\node () at (3.5,-1.5) {\scriptsize $c$};
					\node () at (9.5,-1.5) {\scriptsize $z_{\gamma_1\gamma_2}$};
					\node () at (3.5,2) { $\mu$};
					\node () at (-1.5,2) {\scriptsize $\mathfrak{d}$};	
					
					\draw[->][in=160, out=20]  (3.4,5.5) to node[above] {\scriptsize $u$} (9.5,5.5) ;
					\fill[fill=gray,fill opacity=.25] (-1,-1) rectangle (2,1);
					\fill[fill=gray,fill opacity=.25] (2,1) rectangle (5,3);
					\fill[fill=gray,fill opacity=.25] (5,3) rectangle (8,5);
					\fill[fill=gray,fill opacity=.25] (8,3) rectangle (11,5);
					
					\foreach \y in {10,12,14} \foreach \x in {10,13,16} {\draw (10,\y)--(16,\y); \draw(\x,10)--(\x,14);}
					\fill[fill=gray,fill opacity=.25] (10,10) rectangle (13,12);
					\fill[fill=gray,fill opacity=.25] (13,12) rectangle (16,14);
					\node () at (14.5,11) { $\gamma_2$};
					\draw[->,dotted]   (9.6,5.5) to   (14.5,9.7);
					\node () at (14,8) {\scriptsize $j(z_{\gamma},z_{\gamma_2})$};	
					\node () at (14.5,14.5) {\scriptsize $z_{\gamma_2}$};
					
				\end{tikzpicture}
			\end{center}
			\caption{Composition of cones in the semigroup $\sncd$.}
			\label{figcomp}	
		\end{figure}	
		
		Finally, to see that $ \sncd $ is regular, let $\gamma\in  \sncd $ be a cone with vertex $z_\gamma$. By definition of a connected category, there is an idempotent cone $\epsilon(z_\gamma,\mathfrak{d}')$ in the semigroup $ \sncd $ with vertex $z_\gamma$. Using Lemma \ref{lemud1}, we can write $\gamma=\epsilon(c',\mathfrak{d}_\gamma)\ast u_\gamma$ where $c'\in v\mathcal{C}$, $\mathfrak{d}_\gamma:=R_{\gamma}$ and $u_\gamma$ an isomorphism from $c'$ to $z_\gamma$. Then let $\chi:= \epsilon(z_\gamma,\mathfrak{d}')\ast u_\gamma^{-1}$ so that $z_{\chi}=c'$ and $R_{\chi}=\mathfrak{d}'$. Since $\mathfrak{d}',\mathfrak{d}_\gamma\in \mathfrak{D}$, by Lemma \ref{lemud1} the cone $\chi \in  \sncd $. Also, observe that using (\ref{eqnbin0}), we have $$\gamma(c')=(\epsilon(c',\mathfrak{d}_\gamma)\ast u_\gamma) (c')= \epsilon(c',\mathfrak{d}_\gamma) (c') u_\gamma =1_{c'} u_\gamma=u_\gamma.$$ 
		Similarly we have $\chi(z_\gamma)=u_\gamma^{-1}$, and both $\gamma(c')$ and $\chi(z_\gamma)$ isomorphisms. Then
		$$\gamma\chi\gamma=(\gamma\ast (\chi(z_\gamma))^\circ) \ast (\gamma(c'))^\circ=(\gamma\ast\:u_\gamma^{-1})\ast\:u_\gamma=\gamma\ast\:(u_\gamma^{-1}\:u_\gamma)=\gamma.$$
		Similarly,  $\chi\gamma\chi=\chi$, and so $\chi$ is an inverse of $\gamma$ (see Figure \ref{figinv}). Hence $ \sncd $ is a regular subsemigroup of $ \snc $.	
		
		\begin{figure}[ht]
			\begin{center}
				\begin{tikzpicture}[scale=.5]
					\foreach \x in {-1,1,3,5} \foreach \y in {-1,2,5,8,11} {\draw (-1,\x)--(11,\x); \draw(\y,-1)--(\y,5);}
					\node () at (0.5,4) {\footnotesize $\epsilon(c',\mathfrak{d}_\gamma)$};
					\node () at (0.5,-1.5) {\footnotesize $c'$};
					\node () at (9.5,4) {\footnotesize $\gamma$};			
					\node () at (0.5,0) {\footnotesize $\chi$};
					\node () at (9.5,0) {\footnotesize $\epsilon(z_\gamma,\mathfrak{d}')$};			
					\node () at (9.5,-1.5) {\footnotesize $z_\gamma$};
					\node () at (11.5,0) {\footnotesize $\mathfrak{d}'$};	
					\node () at (11.5,4) {\footnotesize $\mathfrak{d}_\gamma$};	
					
					\draw[->][in=160, out=20]  (0.5,5.5) to node[above] {\footnotesize $u_\gamma$} (9.5,5.5) ;
					\draw[->][in=-20, out=200]   (9.5,-2) to node[below] {\footnotesize $u_\gamma^{-1}$} (0.5,-2);
					\fill[fill=gray,fill opacity=.25] (-1,3) rectangle (2,5);
					\fill[fill=gray,fill opacity=.25] (2,1) rectangle (5,3);
					\fill[fill=gray,fill opacity=.25] (5,1) rectangle (8,3);
					\fill[fill=gray,fill opacity=.25] (8,-1) rectangle (11,1);
					
				\end{tikzpicture}
			\end{center}
			\caption{Locating an inverse $\chi$ of a cone $\gamma$ in the semigroup $\sncd$.}
			\label{figinv}
			
		\end{figure}	
	\end{proof}

	\begin{pro}\label{proud3}
		Let  $\cod$ be  a connected category. The semigroup	$ \sncd $ is left reductive.
	\end{pro}
	\begin{proof}
		Recall that, to prove that $ \sncd $ is left reductive, given $\gamma_1,\gamma_2 \in  \sncd $ such that 
		\begin{equation}\label{eqnass}
			\gamma\gamma_1=\gamma\gamma_2\text{ , for every }\gamma\in  \sncd,
		\end{equation}
		we need to show $\gamma_1=\gamma_2$. Observe that by Definition \ref{dfnnormcone}, to prove that two cones $\gamma_1$ and $\gamma_2$  are equal in a category $\mathcal{C}$, it suffices to show that: 
		\begin{enumerate}[(i)]
			\item the vertices of the cones are same, i.e., $z_{\gamma_1}=z_{\gamma_2}$, and,
			\item each component of the cones coincide, i.e., $\gamma_1(c)=\gamma_2(c)\text{, for every }c\in v\mathcal{C}.$ 
		\end{enumerate}
		
		But since the inclusions are unique between two given objects and using Lemma \ref{lemepi} (1), to verify the condition (ii) above, it suffices to show that	the epimorphic components of the respective morphisms are the same, i.e.,
	$$(\gamma_1(c))^\circ=(\gamma_2(c))^\circ\text{ , for every }c\in v\mathcal{C}.$$

		To begin with, observe that since the semigroup $ \sncd $ is regular, there exists an idempotent cone $\epsilon_1\in \sncd $ such that $\gamma_1 \gr \epsilon_1$, and so  $\gamma_1=\epsilon_1 \gamma_1$. Then using the assumption (\ref{eqnass}) and letting $\gamma:=\epsilon_1$, we have $\epsilon_1 \gamma_1 =\epsilon_1 \gamma_2$. So, $\gamma_1=\epsilon_1 \gamma_2$, i.e., $\gamma_1 \lel \gamma_2$. Hence, using Lemma \ref{lemgc} (1), we see that the vertices of the cones satisfy $z_{\gamma_1} \preceq z_{\gamma_2}$. Similarly, using an idempotent $\epsilon_2\in  \sncd $ such that $\gamma_2 \gr \epsilon_2$, we can show that $z_{\gamma_2} \preceq z_{\gamma_1}$. Thus  $z_{\gamma_1} = z_{\gamma_2}$. 
		
		Next, given an arbitrary $c\in v\mathcal{C}$, by definition of a connected category, there is an idempotent cone $\epsilon=\epsilon(c,\mathfrak{d})$ in the semigroup $ \sncd $ with vertex $c$ such that $\epsilon(c)=1_c$. Letting $\gamma:=\epsilon$ in the assumption (\ref{eqnass}), we get $\epsilon\gamma_1=\epsilon\gamma_2$, and so using equation (\ref{eqnbin}) we have  $\epsilon\ast(\gamma_1(c))^\circ=\epsilon\ast(\gamma_2(c))^\circ$. Now, comparing the component of these cones at the object $c$, we obtain $\epsilon(c)(\gamma_1(c))^\circ=\epsilon(c)(\gamma_2(c))^\circ$. However, since $\epsilon(c)=1_c$, we see that $(\gamma_1(c))^\circ=(\gamma_2(c))^\circ$,  for each object $c\in v\mathcal{C}$.  Thereby we conclude that the cones $\gamma_1$ and $\gamma_2$ coincide, and so the regular semigroup $ \sncd $ is left reductive.
	\end{proof}

	Summarising the above discussion, given a connected category $\cod$, we constructed a left reductive regular semigroup $ \sncd $. We shall refer to this semigroup as the \emph{connection semigroup} of the category $\cod$. Now, to take the discussion forward, we need to explore the left and right ideal structure of  $\sncd$. Since $\sncd$ is a regular subsemigroup of the semigroup $\snc$, Green's relations in $\sncd$ are inherited from $\snc$ (see \cite[Proposition 2.4.2]{howie}). So, using Lemma \ref{lemgc}, we have the following.
	
	\begin{lem}\label{lemgrcd}
		Let $\gamma,\delta$ be cones in the connection semigroup $ \sncd $. Then 
		\begin{enumerate}
			\item $\gamma \lel \delta \text{ if and only if }z_{\gamma}\preceq z_{\delta}$, and $\gamma \gl \delta \text{ if and only if }z_{\gamma}= z_{\delta}$;
			\item $\gamma \ler \delta \text{ if and only if }\gamma(z_{\delta})\text{ is an epimorphism such that }\gamma=\delta \ast\gamma(z_{\delta })$. We have $\gamma\: \gr \:\delta$  if and only if   $\gamma=\delta\ast h$ for a unique isomorphism $h$.
		\end{enumerate} 
	\end{lem}
	
	But in our framework, we have a `neater' description of the Green's $\gr$ relation. Observe that Green's relations in $\sncd$ are restrictions of Green's relations in $\snc$ (see Remark \ref{rmkcc}). In what follows, we shall, by a slight abuse of notation, use the same symbol to denote the partial orders on $\sncd$ and $\snc$. Recall that $ (\mathfrak{D},\sqsubseteq) $ is a down-set of the poset $ (\snc /\gr,\sqsubseteq)$. So, for cones $\gamma,\delta \in \sncd $ such that $\gamma\in \mathfrak{d}_1$ and $\delta\in \mathfrak{d}_2$ where $\mathfrak{d}_1$ and $\mathfrak{d}_2$ are elements of the poset $\mathfrak{D}$, we have 
	\begin{equation}
		\gamma \ler \delta \iff R_\gamma \sqsubseteq R_\delta \iff \mathfrak{d}_1 \sqsubseteq \mathfrak{d}_2.
	\end{equation}  
	Thus we readily conclude that the poset $\sncd/\gr$ is order isomorphic to $\mathfrak{D}$. For later use, we denote this order isomorphism 
	by $G$ and record the above observation as a lemma.
	\begin{lem}\label{lemgrcd1}
		Let $\gamma,\delta$ be cones in the connection semigroup $ \sncd $  such that $\gamma\in \mathfrak{d}_1$ and $\delta\in \mathfrak{d}_2$. Then $\gamma \ler \delta$ if and only if $\mathfrak{d}_1 \sqsubseteq \mathfrak{d}_2$. Further $\gamma\: \gr \:\delta$  if and only if   $\mathfrak{d}_1 = \mathfrak{d}_2$.
	\end{lem}
	 
	To simplify notation, we fix $T:=\sncd$, the semigroup of cones in a connected category $\cod$ for the remainder of the section. 
	\begin{rmk}
	As $T$ is regular, the principal right ideals of $T$ form a normal category $\mathbb{R}(T)$ as defined in equation (\ref{eqnrs}). In particular, we have a poset $(v\mathbb{R}(T),\subseteq)$ of objects of the right normal category. Observe that for cones $\gamma,\delta\in T$, we have $\gamma T\subseteq \delta T$ if and only if $R_\gamma \sqsubseteq R_\delta$. So using Lemma \ref{lemgrcd1}, it follows that the poset $(v\mathbb{R}(\sncd),\subseteq)$ of objects of the right normal category is, in fact, order isomorphic to the poset $(\mathfrak{D},\sqsubseteq)$. Hence
	$$(\sncd/\gr,\sqsubseteq) \cong (\mathfrak{D},\sqsubseteq) \cong (v\mathbb{R}(T),\subseteq).$$
	\end{rmk}
	
	Having characterised the right ideal structure of the semigroup $\sncd$ as a poset, we proceed to the left ideals where further structure emerges. By Lemma \ref{lemgrcd} (1), it is clear that the poset $\sncd/\gl$ is order isomorphic to the poset $(v\mathcal{C},\preceq)$. But to completely describe the left ideal structure of the semigroup $\sncd$, we employ normal categories and dive one additional layer deeper. Recall that since $\sncd$ is a regular semigroup, the principal left ideals of $\sncd$ form a normal category $\mathbb{L}(\sncd)$. 
	
	Given $T:=\sncd$, we define a functor $F\colon \mathbb{L}(T)\to \mathcal{C}$ as follows. For $\epsilon\in E(T)$  and a morphism $r(\epsilon_1,\gamma,\epsilon_2)$ in $\mathbb{L}(T)$ such that $\gamma\in \epsilon_1 T \epsilon_2$, let
	\begin{equation}\label{eqnF}
		vF(T \epsilon):= z_\epsilon\:\text{ and }\: F(r(\epsilon_1,\gamma,\epsilon_2)):=\gamma(z_{\epsilon_1})j(z_\gamma,z_{\epsilon_2}),
	\end{equation}
	where $j(z_\gamma,z_{\epsilon_2})$ is the inclusion morphism in $\mathbb{L}(T)$. 
	
	Observe that given a morphism $r(\epsilon_1,\gamma,\epsilon_2)$ in the category $\mathbb{L}(T)$ from $T \epsilon_1$ to $T \epsilon_2$, since $\gamma\in \epsilon_1T$, we have $\gamma\ler\epsilon_1$. So, using Lemma \ref{lemgrcd} (2), we see that $\gamma(z_{\epsilon_1})$ is an epimorphism in $\mathcal{C}$ such that $\gamma=\epsilon_1\ast \gamma(z_{\epsilon_1})$. Also notice that the expression $\gamma(z_{\epsilon_1})j(z_\gamma,z_{\epsilon_2})$ is the {unique canonical factorisation} of the corresponding morphism belonging to $\mathcal{C}(z_{\epsilon_1},z_{\epsilon_2})$. 
	
	\begin{lem}\label{lemf1}
		$F$ is a well-defined functor from the normal category $\mathbb{L}(T)$ to $\mathcal{C}$.
	\end{lem}
	\begin{proof}
		Suppose that $T \epsilon=T \epsilon'$, then by Lemma \ref{lemgrcd} (1) we have $z_\epsilon=z_{\epsilon'}$, and $vF$ is well-defined on objects. To verify that $F$ is well-defined on morphisms, suppose that $r(\epsilon_1,\gamma,\epsilon_2)=r(\epsilon'_1,\gamma',\epsilon'_2)$ in the left category $\mathbb{L}(T)$. Then, from Section \ref{ssecnc}, this equality of morphisms implies $\epsilon_1\gl \epsilon'_1$, $\epsilon_2\gl \epsilon'_2$ and $\gamma=\epsilon_1\gamma'$. By Lemma \ref{lemgrcd} (1), we have $z_{\epsilon_1}=z_{\epsilon'_1}$ and $z_{\epsilon_2}=z_{\epsilon'_2}$. Further, since $\gamma'(z_{\epsilon'_1})$ is an epimorphism, by Lemma \ref{lemepi} (2), we get  $(\gamma'(z_{\epsilon'_1}))^\circ= \gamma'(z_{\epsilon'_1}).$
		Hence, using equations (\ref{eqnbin0}) and (\ref{eqnbin}), we see that
		$$\gamma(z_{\epsilon_1})= (\epsilon_1\gamma')(z_{\epsilon_1})= \epsilon_1(z_{\epsilon_1}) \: (\gamma'(z_{\epsilon_1}))^\circ=1_{z_{\epsilon_1}}(\gamma'(z_{\epsilon'_1}))^\circ=(\gamma'(z_{\epsilon'_1}))^\circ=\gamma'(z_{\epsilon'_1}).$$ 
		Now since every morphism in the category $\mathcal{C}$ has a unique canonical factorisation, we see that the morphism $\gamma(z_{\epsilon_1})j(z_\gamma,z_{\epsilon_2})=\gamma'(z_{\epsilon'_1})j(z_{\gamma'},z_{\epsilon'_2})$ and so $F$ is well-defined. 
		
		To see that $F$ is a functor from $\mathbb{L}(T)$ to $\mathcal{C}$, first observe that given the identity morphism $1_{T\epsilon}=r(\epsilon,\epsilon,\epsilon)$ in the category $\mathbb{L}(T)$, we have $F(r(\epsilon,\epsilon,\epsilon))=\epsilon(z_{\epsilon})j(z_\epsilon,z_\epsilon)=1_{z_\epsilon}$. Hence $F(1_{T\epsilon})= 1_{vF(T\epsilon)}$ and the identities are preserved by $F$. 
		
		Next, let  $r_1:=r(\epsilon_1,\gamma,\epsilon_2)$ and $r_2:=r(\epsilon'_1,\gamma',\epsilon'_2)$ be morphisms in $\mathbb{L}(T)$ such that $\epsilon_2\gl \epsilon'_1$. Then $r_1r_2=r(\epsilon_1,\gamma\gamma',\epsilon'_2)$ is a morphism such that $F(r_1r_2)=\gamma\gamma'(z_{\epsilon_1})j(z_{\gamma\gamma'},z_{\epsilon'_2})$. 
		Also, $F(r_1)=\gamma(z_{\epsilon_1})j(z_{\gamma},z_{\epsilon_2})$ and $F(r_2)=\gamma'(z_{\epsilon'_1})j(z_{\gamma'},z_{\epsilon'_2})$. Thus $z_\gamma\preceq z_{\epsilon'_1}$  since $z_{\epsilon_2}=z_{\epsilon'_1}$,  and so by Definition \ref{dfnnormcone} (1), we have $j(z_\gamma,z_{\epsilon'_1})\gamma'(z_{\epsilon'_1})=\gamma'(z_\gamma)$. Moreover, as $z_{\gamma\gamma'}= \im \gamma'(z_\gamma)$, using canonical factorisation, we can write $\gamma'(z_\gamma)=(\gamma'(z_\gamma))^\circ j(z_{\gamma \gamma'},z_{\gamma'})$. Therefore, the morphisms $F(r_1)$ and $F(r_2)$ are composable and 
		$$F(r_1)F(r_2)=\gamma(z_{\epsilon_1})j(z_{\gamma},z_{\epsilon_2})\gamma'(z_{\epsilon'_1})j(z_{\gamma'},z_{\epsilon'_2})=\gamma(z_{\epsilon_1})\gamma'(z_\gamma)j(z_{\gamma'},z_{\epsilon'_2})=\gamma(z_{\epsilon_1})(\gamma'(z_\gamma))^\circ j(z_{\gamma \gamma'},z_{\gamma'}) j(z_{\gamma'},z_{\epsilon'_2}).$$
		Finally from  $\gamma(z_{\epsilon_1}) (\gamma'(z_{\gamma}))^\circ=\gamma\gamma'(z_{\epsilon_1})$ and $j(z_{\gamma \gamma'},z_{\gamma'}) j(z_{\gamma'},z_{\epsilon'_2})=j(z_{\gamma \gamma'},z_{\epsilon'_2})$, we obtain
		$$F(r_1)F(r_2)=\gamma\gamma'(z_{\epsilon_1})j(z_{\gamma \gamma'},z_{\epsilon'_2})=F(r_1r_2).$$
		Thus the assignment $F$ preserves the composition also, whence $F$ is a functor.
	\end{proof}
	
	\begin{lem}\label{lemf2}
		The functor $F$ is a normal category isomorphism.
	\end{lem}
	\begin{proof}
		By Lemma \ref{lemgrcd} (1), the map $vF$ is clearly a bijection. Given an inclusion $j(T\epsilon_1,T\epsilon_2)$ in the category $\mathbb{L}(T)$, we can easily see that $j(z_{\epsilon_1},z_{\epsilon_2})$ is an inclusion  in the category $\mathcal{C}$. Hence $F$ is inclusion preserving. 
		
		To see that $F$ is faithful, suppose that $F(r(\epsilon_1,\gamma,\epsilon_2))=F(r(\epsilon'_1,\gamma',\epsilon'_2))$, i.e.,  $\gamma(z_{\epsilon_1})j(z_\gamma,z_{\epsilon_2})=\gamma'(z_{\epsilon'_1})j(z_{\gamma'},z_{\epsilon'_2})$ in the category $\mathcal{C}$. Then  $z_{\epsilon_1}=z_{\epsilon'_1}$ and $z_{\epsilon_2}=z_{\epsilon'_2}$ and so by Lemma \ref{lemgrcd} (1), we have $\epsilon_1\gl \epsilon'_1$ and $\epsilon_2\gl \epsilon'_2$. On another hand, using the canonical factorisation property of morphisms in $\mathcal{C}$, we get  $\gamma(z_{\epsilon_1})=\gamma'(z_{\epsilon'_1})$. Then applying  (\ref{eqnbin}), we obtain 
		$$\epsilon_1\gamma' =\epsilon_1\ast (\gamma'(z_{\epsilon_1}))^\circ=\epsilon_1\ast  (\gamma'(z_{\epsilon'_1}))^\circ=\epsilon_1\ast  (\gamma(z_{\epsilon_1}))^\circ=\epsilon_1\gamma=\gamma.$$
		Hence, $r(\epsilon_1,\gamma,\epsilon_2)=r(\epsilon'_1,\gamma',\epsilon'_2)$ in the category $\mathbb{L}(T)$, and so $F$ is faithful.
		
		To show that $F$ is full, given a morphism $f\in\mathcal{C}(c_1,c_2)$, let $f=qj$ be its canonical factorisation. Since $\sncd$ is a connected category, there exist idempotent cones $\epsilon_1\text{ and }\epsilon_2\in E(\sncd)$ with vertices $c_1$ and $c_2$, respectively. Let $\gamma:=\epsilon_1\ast q$. Then since inclusions are unique and observing that $z_\gamma$ is the domain of $j$ and $z_{\epsilon_2}=c_2$, we see that $j(z_\gamma,z_{\epsilon_2})=j$. Now using (\ref{eqnbin0}), we also observe that $\gamma(z_{\epsilon_1}) = \epsilon_1\ast q (c_1)= \epsilon_1(c_1) q=1_{c_1}q=q$. By Lemma \ref{lemud2}, the cone $\gamma$ is in $\sncd$ and by Lemma \ref{lemgrcd}, we have $\gamma\in \epsilon_1T\epsilon_2$. So $r:=r(\epsilon_1,\gamma,\epsilon_2)$ is a morphism in the normal category $\mathbb{L}(T)$ such that $F(r)=\gamma(z_{\epsilon_1})j(z_\gamma,z_{\epsilon_2}) =  qj=f$. We conclude that $F$ is a normal category isomorphism, as required.
	\end{proof}
	We now summarise the conclusions of Lemma \ref{lemgrcd1} and Lemma \ref{lemf2} as follows.
	\begin{pro}\label{procatiso}
		Given a connected category $\cod$, the regular poset $\sncd/\gr$ is order isomorphic to $\mathfrak{D}$ and the left ideal normal category $\mathbb{L}(\sncd)$ is isomorphic to $\mathcal{C}$. 	
	\end{pro}
	
	Now, given a connected category $\cod$, in particular,  taking $\mathfrak{D}= \snc /\gr$, we obtain $ \sncd = \snc $. Observe that here we are treating an arbitrary normal category as a connected category. So by applying Proposition \ref{proud3}, we deduce that the semigroup $ \snc $ of all cones in a normal category $\mathcal{C}$, is indeed left reductive. Further applying Proposition \ref{procatiso}, we have the following corollary, which we believe is of independent interest.
	
	\begin{cor}\label{cortlclr}
		The semigroup $ \snc $ of all cones in a normal category $\mathcal{C}$ is  left reductive. Moreover,  the left category $\mathbb{L}(\snc)$ is normal and isomorphic to $\mathcal{C}$.
	\end{cor}
	
	The latter half of the above corollary is already known, and an alternative proof can be found in \cite[Section III.3.3]{cross}, where the functor is defined from $\mathcal{C}$ to $\mathbb{L}(\snc)$. We shall return to this discussion on normal categories in Section \ref{sstlrs}.
	
	We now continue our journey towards the main theorem of the paper. To this end, Proposition \ref{procatiso} gives us the following useful characterisation of the quasi-orders on the set of idempotents of $\sncd$ as follows. 
	\begin{lem}\label{lemidgrcd}
		Let $\epsilon_1=\epsilon(c_1,\mathfrak{d}_1)$ and $\epsilon_2=\epsilon(c_2,\mathfrak{d}_2)$ be idempotents in the semigroup $\sncd$, then
		\begin{enumerate}
			\item $\epsilon_1\ler\epsilon_2\text{ if and only if }\mathfrak{d}_1\sqsubseteq \mathfrak{d}_2$,
			\item $\epsilon_1\lel\epsilon_2 \text{ if and only if } c_1\preceq c_2$.
		\end{enumerate} 
	\end{lem}
	
	
	\subsection{Structure of left reductive regular semigroups}\label{ssecslrs}
	
	In the previous section, we saw a left reductive regular semigroup constructed using a connected category. Now, we proceed to show how a left reductive regular semigroup gives rise to a connected category. To this end, recall from Theorem \ref{thmlslr} that given a left reductive regular semigroup $S$, we have $S\xhookrightarrow{} \sls$ via the map $\bar{\rho}\colon a\mapsto  r^a$. In fact, more is true. 
	\begin{figure}[h]
		\begin{center}
			\begin{tikzpicture}[scale=.4]
				\foreach \y in {9,11,13,15} \foreach \x in {-6,-3,0,3,6} {\draw (-6,\y)--(6,\y); \draw(\x,9)--(\x,15);}
				\fill[fill=gray,fill opacity=.25] (-6,13) rectangle (-3,15);
				\fill[fill=gray,fill opacity=.25] (-3,13) rectangle (0,11);
				\fill[fill=gray,fill opacity=.25] (0,11) rectangle (3,9);
				\fill[fill=gray,fill opacity=.25] (3,9) rectangle (6,11);
				
				\node () at (-4.5,14) { $e$};
				\node () at (1.5,14) { $x$};
				\node () at (1.5,10) { $f$};
				\node () at (4.5,14) { $a$};
				
				\draw[->,dotted][in=160, out=20]  (1.5,16) to node[above] {\scriptsize $\bar{\rho}$} (20.5,16) ;
				\draw[dashed]  (9.5,16) to  (9.5,5) ;
				\node () at (0,8) { $S$};
				\node () at (19,5.5) { $\sls$};

				\foreach \y in {7,9,11,13,15} \foreach \x in {13,16,19,22,25} {\draw (13,\y)--(25,\y); \draw(\x,7)--(\x,15);}
				\fill[fill=gray,fill opacity=.25] (13,13) rectangle (16,15);
				\fill[fill=gray,fill opacity=.25] (16,13) rectangle (19,11);
				\fill[fill=gray,fill opacity=.25] (19,11) rectangle (25,7);
				\node () at (14.5,14) { $r^e$};
				\node () at (20.5,14) { $\gamma$};
				\node () at (20.5,10) { $r^f$};
				\node () at (23.5,14) { $r^a$};
				\draw[->][in=165, out=15]  (14.5,14.7) to node[above] {\scriptsize $r(e,x,f)$} (20.5,14.7) ;
				
			\end{tikzpicture}
		\end{center}
		\caption{The injective homomorphism $\bar{\rho}$ for a left reductive regular semigroup $S$.}
		\label{figrho}	
	\end{figure}	
	
	\begin{lem}
		If $S$ is a left reductive regular  semigroup, for an element $a\in S$, the map $\bar{\rho}_{|R_a}$ is a bijection onto $R_{r^a}$, where $R_{r^a}$ denotes the $\gr$-class of the cone $r^a$ in the semigroup $\sls$.
	\end{lem}
	\begin{proof}
		We have already seen that the map $\bar{\rho}$ is injective, when $S$ is a left reductive regular  semigroup. To see that the map $\bar{\rho}_{|R_a}$ has image $ R_{r^a}$, let $\gamma\in R_{r^a}\subseteq\sls$. Thus $\gamma \gr r^a$ and $z_\gamma=Sf$ for some $f\in E(S)$.	Then as $S$ is regular, there exists $e\in E(S)$ such that $e\gr a$ in $S$. But since $\bar{\rho}$ is an injective homomorphism, we have  $r^e\gr r^a \gr \gamma$ in the semigroup $\sls$. Now, by Lemma \ref{lemgc} (2), there is an isomorphism in $\los$, say $r(e,x,f)$ such that $\gamma=r^e\ast r(e,x,f)=r^x$. Then since $r(e,x,f)$ is an isomorphism in $\los$,  we have $e\gr x$ in the semigroup $S$ (see Figure \ref{figrho} and \cite[Proposition III.13(c)]{cross}). So, we obtain $x\gr a$ and $\bar{\rho}(x)=r^x=\gamma$, whence $\bar{\rho}_{|R_a}$ maps onto $ R_{r^a}$ .
	\end{proof}
	
	In the sequel, we shall also use $\bar{\rho}$ to denote the induced map (i.e. $\bar{\rho}_{|R_a}$) on the $\gr$-classes of $S$. From the discussion above, we see that $\bar{\rho}$: $S/\gr \:\xhookrightarrow{}\: \sls/\gr$, where $ R_{a}$ maps to $ R_{r^a}$, is  such that, for any element $a\in S$, the sets $R_a$ and $R_{r^a}$ are in bijection. For ease of notation in the sequel,  we shall denote the $\gr$-class in $\sls$ containing the cone $r^a$ by just $\mathfrak{r}_a$. Observe that there might be $\gr$-classes in $\sls$ which are not of the form $\mathfrak{r}_a$ for some $a\in S$. For instance, in the Figure \ref{figrho}, the bottom $\gr$-class of the semigroup $\sls$ does  not have a preimage under $\bar{\rho}$. As the reader may have already realised (also see Remark \ref{rmkcc}), the definition of a connection semigroup involves excluding from the semigroup $\sls$ the $\gr$-classes that are not of the form $\mathfrak{r}_a$. So, let 
	\begin{equation}\label{eqnr}
		\mathfrak{R}:=\bar{\rho}(S/\gr)=\{ R_{ r^e}: e\in E(S) \}=\{\mathfrak{r}_e : e\in E(S)\}.
	\end{equation}
	
	Then, $\mathfrak{R}\subseteq \sls/\gr$ and  for each $Sf\in v\los$, there is $\mathfrak{r}_f=\bar{\rho}(R_f)\in  \mathfrak{R}$ such that $\mathfrak{r}_f$ contains an idempotent cone with vertex $Sf$, namely $ r^f$. Observe that this idempotent cone may be denoted by $\epsilon(Sf,\mathfrak{r}_f)$, i.e., the object $Sf$ is connected by $\mathfrak{r}_f$. So, we can get a connected category  from a left reductive regular semigroup. Hence we have proved the following proposition:
	\begin{pro}\label{prolsr}
		Let $S$ be a left reductive regular  semigroup. The normal category $\los$ is connected by the regular poset $\mathfrak{R}$, that is, $\lr$ is a connected category.
	\end{pro}
	
	Given the connected category $\lr$, by Propositions \ref{proud2} and \ref{proud3}, we know that the connection semigroup $\slr$ is a left reductive regular  semigroup. Moreover, by equation (\ref{eqnesncd}),  the idempotents of $\slr$ are given by:
	$$E(\slr)=\{ \epsilon(Se,\mathfrak{r}_e) : e\in E(S) \} = \{ r^e  : e\in E(S) \}.$$

	\begin{pro}\label{prolsriso}
		Given a left reductive regular  semigroup $S$, the connection semigroup $$\slr =\{ r^a:a\in S \}$$ is isomorphic to $S$.
	\end{pro}
	\begin{proof}
		First, since $S$ is regular and left reductive, using Theorem \ref{thmlslr} there is an injective homomorphism $\bar{\rho}:S \to \sls$ given by $a\mapsto r^a$.  Now given any $x\in S$, for some $e\in E(R_x)$ and $f\in E(L_x)$, we have $r^x = r^e\ast r(e,x,f)$. Here $r^e=\epsilon(Se,\mathfrak{r}_e)$ is an idempotent cone in $\slr$ and $r(e,x,f)$ is an isomorphism in $\los$. Now, using Lemma \ref{lemud1}, we see that $\bar{\rho}:S \to \slr$ is an injective homomorphism  which is also surjective (see Figure \ref{figrho}). Hence the connection semigroup $\slr$ is isomorphic to $S$.
	\end{proof}
	
	\subsection{Category equivalence}\label{sseccateqlrs}
	
	We have seen that given a connected category $\cod$, we get a left reductive regular  semigroup $\sncd$ such that the category $\mathbb{L}(\sncd)$ is isomorphic to $\mathcal{C}$ and the down-set $\sncd/\gr$ is isomorphic to the regular poset $\mathfrak{D}$ (Proposition \ref{procatiso}). Conversely, given a left reductive regular  semigroup $S$, we have obtained a connected category $\lr$ such that its connection semigroup $\slr$ is isomorphic to $S$ (Proposition \ref{prolsriso}). Next, we proceed to extend this correspondence to a category equivalence.
	
	First observe that left reductive regular semigroups  form a full subcategory, say $\mathbf{LRS}$ of the category $\mathbf{RS}$ of regular semigroups, with semigroup homomorphisms as morphisms. 
	
	\begin{dfn}\label{dfnccmor}
		Given connected categories $\cod$ and $\codp$, we define a \emph{CC-morphism} as an ordered pair $m:=(F,G)$  such that $F\colon \mathcal{C}\to\mathcal{C'}$ is an inclusion preserving functor and $G\colon \mathfrak{D}\to\mathfrak{D'} $ is an order preserving map satisfying:
		\begin{equation}\label{eqnccmor}
			c\text{ is connected to }\mathfrak{d} \implies F(c) \text{ is connected to }G(\mathfrak{d}) \quad\text{ and }\quad F(\epsilon(c,\mathfrak{d})(c'))=\epsilon(F(c), G(\mathfrak{d})) (F(c'))
		\end{equation}
		for every $c'\in v\mathcal{C}$.
	\end{dfn}  
	
	\begin{rmk}\label{rmkccmor1}
		Given a CC-morphism $m:=(F,G)$ from  $\cod$ to $\codp$, by definition the functor $F$ maps an idempotent cone $\epsilon:=\epsilon(c,\mathfrak{d})$ in the category $\mathcal{C}$ to the idempotent cone $\epsilon':=\epsilon(F(c),G(\mathfrak{d}))$ in the category $\mathcal{C'}$. This makes the relation between the categories $\mathcal{C}$ and $\mathcal{C'}$ via the functor $F$ rather strong. Roughly speaking, the semigroup homomorphism associated with the morphism $m$ will be an `extension' of this mapping $\epsilon \mapsto \epsilon'$.
	\end{rmk}
	
	\begin{rmk}\label{rmkccmor2}
		It is clear that any notion of morphism between connected categories must respect the connection between $\mathcal{C}$ and $\mathfrak{D}$. In fact, given an arbitrary pair $m:=(F,G)$ from  $\cod$ to $\codp$ such that $F$ is a normal category isomorphism from $\mathcal{C}$ to $\mathcal{C'}$ and $G$ is an order isomorphism from   $\mathfrak{D}$ to $\mathfrak{D'} $, we can construct examples such that the semigroups $\widehat{\cod}$ and $\widehat{\codp}$ are not isomorphic. Hence condition (\ref{eqnccmor}) is crucial in the definition of a CC-morphism, to obtain isomorphic semigroups.
	\end{rmk}
	
	It is routine to verify that the class of all connected categories with CC-morphisms form a category. This category will be denoted by $\mathbf{CC}$ in the sequel.
	
	Recall from Proposition \ref{prolsr} that given an object $S$ in $\mathbf{LRS}$, the connected category $\lr$ is an object in the category $\mathbf{CC}$. Now we proceed to make this correspondence functorial. 
	\begin{lem}\label{lemmorC}
		Given a semigroup homomorphism $\phi \colon S \to S'$ in the category $\mathbf{LRS}$, define a functor $F_\phi\colon\los\to\losp$ and a map $G_\phi\colon\mathfrak{R}\to\mathfrak{R'}$ as follows: for idempotents $e,f\in S$ and $u\in eSf$,
		$$ vF_\phi: Se \mapsto S'(e\phi) \:,\quad F_\phi\colon r(e,u,f) \mapsto r(e\phi,u\phi, f\phi) \:\text{ and }\: G_\phi\colon \mathfrak{r}_e \mapsto \mathfrak{r}_{e\phi}.$$
		Then $m_\phi:=(F_\phi,G_\phi)$ is a CC-morphism from $\lr$ to $\lrp$.
	\end{lem}
	\begin{proof}
		It is straightforward to verify that $F_\phi$ is an inclusion preserving functor from $\mathbb{L}(S)$ to $\mathbb{L}(S')$ and $G_\phi$ is an order preserving map from $\mathfrak{R}$ to $\mathfrak{R'}$ (see equation (\ref{eqnr})). To verify (\ref{eqnccmor}), let the object $Se$ be connected to $\mathfrak{r}_e$ in $\lr$, we are taking $e$ to be an idempotent in $S$  such that $r^e$ is an idempotent in $\slr$. Let $e':=e\phi$. As $\phi$ is a homomorphism,  $e'$ is an idempotent in $S'$. Hence in the category $\lrp$, we have $S'e'$  connected to $\mathfrak{r}_{e'}$, where $\mathfrak{r}_{e'}$ denotes the $\gr$-class $R_{r^{e'}}$ of the cone $r^{e'}$ in the semigroup $\widehat{\losp}$. Further, for every $Sf \in v\los$, observe that $\epsilon(Se,\mathfrak{r}_e)(Sf)=r^e(Sf)=r(e,ef,f)$. Since $\phi$ is a homomorphism we have $(ef)'=e'f'$ so that $$F_\phi(\epsilon(Se,\mathfrak{r}_e)(Sf))=  F_\phi(r(e,ef,f))=r(e',(ef)',f')=r(e',e'f',f').$$ 
		Also,  $\epsilon (F_\phi(Se),G_\phi(\mathfrak{r}_e)) =\epsilon (S'e',\mathfrak{r}_{e'})=r^{e'}$ in $\losp$. Hence
		$$\epsilon(F_\phi(Se),G_\phi(\mathfrak{r}_e))(F_\phi(Sf))=r^{e'}(S'f')=r(e',e'f',f').$$
		Therefore $m_\phi:=(F_\phi,G_\phi)$ satisfies equation (\ref{eqnccmor}) and so $m_\phi$ is a CC-morphism.
	\end{proof} 
	Further, it is routine to verify that this assignment preserves identities and compositions. Hence, we have the following proposition.
	\begin{pro} \label{profunctorC}
		The assignment
		$$S\mapsto\lr\: \text{ and } \: \phi \mapsto m_\phi=(F_\phi,G_\phi) $$
		constitutes a functor $\mathtt{C}$ from the category $\mathbf{LRS}$ of left reductive regular semigroups to the category $\mathbf{CC}$ of connected categories.
	\end{pro}
	
	To build a functor in the opposite direction, we have seen that given a connected category $\cod$ in $\mathbf{CC}$, by Propositions \ref{proud2} and \ref{proud3}, the semigroup $\sncd\in \mathbf{LRS}$. Now, given a CC-morphism in $\mathbf{CC}$, we need to construct a semigroup homomorphism. To this end, recall that given a connected category $\cod$, for each $c\in v\mathcal{C}$, there is an associated idempotent cone $\epsilon=\epsilon(c,\mathfrak{d})$ in $\sncd$ such that $R_{\epsilon(c,\mathfrak{d})}=\mathfrak{d}$. By Lemma \ref{lemud1}, every cone in $\sncd$ may be written as $\epsilon\ast u$, for an idempotent cone $\epsilon$ and an isomorphism $u$ in $\mathcal{C}$. Let $m:=(F,G)$ be a morphism in $\mathbf{CC}$ from $\cod$ to $\codp$. Define $\phi_m\colon \sncd \to \sncdp$ by
	\begin{equation}\label{eqnhomo}
		\phi_m\colon \epsilon(c,\mathfrak{d})\ast u\mapsto \epsilon(F(c),G(\mathfrak{d})) \ast F(u).
	\end{equation}
	
	\begin{lem}\label{lemphi1}
		$\phi_m$ is a well-defined map from semigroup $\sncd$ to $\sncdp$. 
	\end{lem}
	\begin{proof}
		First, observe that by Lemma \ref{lemud1}, any cone in $\sncd$ admits a representation, not necessarily unique, of  the form $ \epsilon(c,\mathfrak{d})\ast u$. Then $\epsilon(F(c),G(\mathfrak{d}))$ will  be an idempotent cone in $\sncdp$. Note that an inclusion preserving functor $F$ preserves normal factorisations (see \cite[Proof of Lemma V.4]{cross} for a routine verification). So $F(u)$ will be an isomorphism and using Lemma \ref{lemud1}, we see that $\epsilon(F(c),G(\mathfrak{d}))\ast F(u)$ will be a cone in $\sncdp$. Now let $\gamma$ be a cone in $\sncd$ with vertex $c$ and $R_\gamma=\mathfrak{d}$. As in Remark \ref{rmkeu}, suppose $\gamma=\epsilon_1\ast u_1=\epsilon_2\ast u_2$ where $\epsilon_1:=\epsilon(c_1,\mathfrak{d})$ and $\epsilon_2:=\epsilon(c_2,\mathfrak{d})$ are idempotent cones, and $u_1:=\gamma(c_1)$ and $u_2:=\gamma(c_2)$ are isomorphisms. Let $c_1':=F(c_1)$, $c_2':=F(c_2)$, $\mathfrak{d}':=G(\mathfrak{d})$, $\epsilon_1':=\epsilon(c_1',\mathfrak{d}')$, $\epsilon_2':=\epsilon(c_2',\mathfrak{d}')$, $u_1':=F(u_1)$ and $u_2':=F(u_2)$. We need to show that $\epsilon_1'\ast u_1' = \epsilon_2'\ast u_2'$, whence $(\epsilon_1\ast u_1)\phi_m =(\epsilon_2\ast u_2)\phi_m$.
		
		Since $ \gamma  \gr \epsilon_1$, by Lemma \ref{lemgrcd} (2), we have $ \gamma=\epsilon_1\ast \gamma(c_1)$. So using equation (\ref{eqnbin0}), we get $ \gamma(c_2)=\epsilon_1(c_2) \gamma(c_1)$. As  $m$ is a CC-morphism, using equation (\ref{eqnccmor}), 
		$$u_2'=F(\gamma(c_2))=F(\epsilon_1(c_2) \gamma(c_1))= F(\epsilon(c_1,\mathfrak{d})(c_2)) F( \gamma(c_1)) =\epsilon(F(c_1),G(\mathfrak{d}))(F(c_2))=\epsilon_1'(c_2') u_1'.$$
		
		Finally,  by Lemma \ref{lemidgrcd} (2) we have $\epsilon_1' \gr \epsilon_2'$, and so by Lemma \ref{lemgrcd} we obtain $\epsilon_1'=\epsilon_2'\ast \epsilon_1'(c_2')$, whence $$\epsilon_1'\ast u_1'= \epsilon_2'\ast \epsilon_1'(c_2')\ast u_1'=\epsilon_2'\ast \epsilon_1'(c_2')  u_1'=\epsilon_2'\ast u_2'.$$ 
		Thus $\phi_m$ is a well-defined map from the  semigroup $\sncd$ to the semigroup $\sncdp$. 
	\end{proof}

	\begin{lem}\label{lemphi2}
		$\phi_m$ is a semigroup homomorphism. 
	\end{lem}
	\begin{proof}
		Using Lemma \ref{lemud1}, let $\gamma_1:=\epsilon(c_1,\mathfrak{d}_1)\ast \gamma_1(c_1)=\epsilon_1\ast u_1$ and $\gamma_2:=\epsilon(c_2,\mathfrak{d}_2)\ast \gamma_2(c_2)=\epsilon_2\ast u_2$ be two arbitrary cones in $\sncd$ such that  $u_1:=\gamma_1(c_1)\colon c_1\to z_{\gamma_1}$ and $u_2:=\gamma_2(c_2)\colon c_2\to z_{\gamma_2}$ are isomorphisms. We need to show that $(\gamma_1  \gamma_2)\phi_m=(\gamma_1\phi_m)  (\gamma_2\phi_m)$. 
		
		To this end, given a morphism $u\colon c_1\to c_2$ in the category $\mathcal{C}$, we shall denote its image in $\mathcal{C'}$ under $F$ by dashed versions without further comment, i.e., we shall denote  $F(u)\colon F(c_1) \to F(c_2)$ by $u'\colon c_1' \to c_2'$. 
		Then since $\gamma_1(c_1)=(\gamma_1(c_1))^\circ$,  we have 
		$$\gamma_1  \gamma_2 
		= (\epsilon_1\ast \gamma_1(c_1)) \: (\epsilon_2\ast \gamma_2(c_2)) 
		= (\epsilon_1\ast \gamma_1(c_1))\ast (\epsilon_2(z_{\gamma_1})\ast \gamma_2(c_2))^\circ= \epsilon_1\ast (\gamma_1(c_1))^\circ\ast (\epsilon_2(z_{\gamma_1})\ast \gamma_2(c_2))^\circ.$$
		Now, using Lemma \ref{lemepi} (3), we see that $$(\gamma_1(c_1))^\circ\ast (\epsilon_2(z_{\gamma_1})\ast \gamma_2(c_2))^\circ= (\gamma_1(c_1)\: \epsilon_2(z_{\gamma_1}) \: \gamma_2(c_2))^\circ.$$ 
		So if we let $qu:=(\gamma_1(c_1)\: \epsilon_2(z_{\gamma_1}) \: \gamma_2(c_2))^\circ$ and $\epsilon(c,\mathfrak{d}):=\epsilon_1\ast q$, we obtain
		$$\gamma_1  \gamma_2= \epsilon_1\ast (\gamma_1(c_1)\: \epsilon_2(z_{\gamma_1}) \: \gamma_2(c_2))^\circ
		= \epsilon_1\ast qu
		= \epsilon(c,\mathfrak{d})\ast u.	$$
		Hence by definition of $\phi_m$, we get  
		$$(\gamma_1  \gamma_2)\phi_m=(\epsilon(c,\mathfrak{d})\ast u)\phi_m= \epsilon(F(c),G(\mathfrak{d}))\ast F(u) =\epsilon(c',\mathfrak{d}')\ast u'.$$
		Then as $F$ is a functor, using equation (\ref{eqnccmor}), we reach
		$$	\epsilon(c',\mathfrak{d}')\ast u'=\epsilon_1'\ast q'\ast u'=\epsilon_1'\ast F(qu).$$
		Also, as $F$ preserves normal factorisations, using Lemma \ref{lemepi} (3) and equation (\ref{eqnccmor}), we conclude that
		$$F(qu)= F((\gamma_1(c_1)\: \epsilon_2(z_{\gamma_1}) \: \gamma_2(c_2))^\circ) = (F(\gamma_1(c_1)\: \epsilon_2(z_{\gamma_1}) \: \gamma_2(c_2)))^\circ =(u_1'\:\epsilon_2'(z_{\gamma_1}')\: u_2')^\circ= u_1'\ast(\epsilon_2'(z_{\gamma_1}')\ast u_2')^\circ.$$
		Therefore putting everything together and using equation (\ref{eqnbin}), 
		$$		(\gamma_1  \gamma_2)\phi_{m} =\epsilon_1'\ast F(qu)=\epsilon_1'\ast u_1'\ast(\epsilon_2'(z_{\gamma_1}')\ast u_2')^\circ=(\epsilon_1'\ast u_1')\: (\epsilon_2'\ast u_2')=(\gamma_1\phi_m)  (\gamma_2\phi_m),$$
		as required.
	\end{proof}
	
	After having constructed a semigroup homomorphism from a CC-morphism, we may now  routinely verify the following assertion.
	
	\begin{pro}\label{profunctorS}
		The assignment
		$$\cod\mapsto\sncd\: \text{ and } \: m \mapsto \phi_m $$
		constitutes a functor $\mathtt{S}$ from the category $\mathbf{CC}$ of connected categories to the category $\mathbf{LRS}$ of left reductive regular semigroups.
	\end{pro}

Now, we have all the ingredients to prove our main theorem: the category equivalence of $\mathbf{LRS}$ and $\mathbf{CC}$. To establish this equivalence, we need functors $\mathtt{C}\colon\mathbf{LRS}\to\mathbf{CC}$ and $\mathtt{S}\colon\mathbf{CC}\to\mathbf{LRS}$ with natural isomorphisms $1_{\mathbf{LRS}}\cong \mathtt{C}\mathtt{S}$ and $1_{\mathbf{CC}} \cong \mathtt{S}\mathtt{C}$. Having already constructed the functors $\mathtt{C}$ and $\mathtt{S}$ in Propositions \ref{profunctorC} and \ref{profunctorS}, we proceed to prove the natural isomorphisms.
\begin{lem}\label{lemcatiso1}
	The identity functor  $1_{\mathbf{LRS}}$ is naturally isomorphic to the functor  $\mathtt{C}\mathtt{S}\colon\mathbf{LRS}\to \mathbf{LRS}$.	
\end{lem}
\begin{proof}
	To prove the lemma, we need to illustrate a natural transformation $\eta$ between the functors $1_{\mathbf{LRS}} $ and $\mathtt{C}\mathtt{S} $ such that each of its components is a semigroup isomorphism. That is, for each object $S$ in the category $\mathbf{LRS}$, we need to assign an isomorphism  $\eta(S)\colon 1_{\mathbf{LRS}}(S)\to\mathtt{C}\mathtt{S}(S)$ in $\mathbf{LRS}$ such that given an arbitrary morphism $\phi\colon S \to S'$ in $\mathbf{LRS}$, the following diagram commutes.
	\begin{equation*}\label{}
		\xymatrixcolsep{3pc}\xymatrixrowsep{4pc}\xymatrix
		{
			S \ar[rr]^{\eta(S)} \ar[d]_{\phi}
			&& \lr \ar[d]^{\mathtt{C}\mathtt{S}(\phi)} \\
			S' \ar[rr]^{\eta(S')} && \mathbb{L}(S')_\mathfrak{R}
		}
	\end{equation*}
	Observe that
	$$1_{\mathbf{LRS}}(S)=S \quad\text{ and }\quad \mathtt{C}\mathtt{S}(S):= \mathtt{S}(\mathtt{C}(S))= \mathtt{S}(\lr)= \slr.$$
	    
	By Proposition \ref{prolsriso} we know that the map $\bar{\rho}\colon a \mapsto r^a$ is an isomorphism from $S$ to $\slr$. So, for each object $S\in \mathbf{LRS}$, we define this morphism $\eta(S):=\bar{\rho}(S)$.  Then we see that $\bar{\rho}(S)$ is a semigroup isomorphism from $1_{\mathbf{LRS}} (S)$ to  $\mathtt{C}\mathtt{S}(S)$. Further, given a semigroup homomorphism $\phi\colon S \to S'$ in $\mathbf{LRS}$, we can routinely verify that the following diagram commutes.
	\begin{equation*}\label{}
		\xymatrixcolsep{3pc}\xymatrixrowsep{4pc}\xymatrix
		{
			S \ar[rr]^{\bar{\rho}(S)} \ar[d]_{\phi}
			&& \lr \ar[d]^{\mathtt{C}\mathtt{S}(\phi)} \\
			S' \ar[rr]^{\bar{\rho}(S')} && \mathbb{L}(S')_\mathfrak{R}
		}
	\end{equation*}
	Hence the assignment $S\mapsto \bar{\rho}(S)$ constitutes a natural isomorphism between the functors $1_{\mathbf{LRS}}$ and $\mathtt{C}\mathtt{S}$.
\end{proof}

\begin{lem}\label{lemcatiso2}
	The identity functor  $1_{\mathbf{CC}}$ is naturally isomorphic to the functor  $\mathtt{S}\mathtt{C}\colon\mathbf{CC}\to \mathbf{CC}$.	
\end{lem}
\begin{proof}
	From Proposition \ref{procatiso}, given a connected category $\cod$, the left ideal category $\mathbb{L}(\sncd)$  is a normal category isomorphic to the category $\mathcal{C}$ and the regular poset  $\sncd/\gr$ is order isomorphic to   $\mathfrak{D}$, via the functor $F$ and the map $G$ respectively. Further, we can routinely (but admittedly a bit cumbersome in terms of notation) verify  that $\iota:=(F,G)$ satisfies condition (\ref{eqnccmor}), and both $F$ and $G$ are bijections; hence $\iota$ is a CC-isomorphism. So, for each object $\cod\in \mathbf{CC}$, if we denote this morphism by $\iota(\cod)$, then since 
	$$1_{\mathbf{CC}}(\cod)=\cod \quad\text{ and }\quad \mathtt{S}\mathtt{C}(\cod):=\mathtt{C}(\mathtt{S}(\cod))= \mathtt{C}(\sncd)= \mathbb{L}(\sncd)_\mathfrak{R}$$ 
	we see that $\iota(\cod)$ is a CC-isomorphism from $1_{\mathbf{CC}} (S)$ to  $\mathtt{S}\mathtt{C}(\cod)$. Further, given a CC-morphism $m\colon \cod \to \codp$, we can routinely verify that the following diagram in the category $\mathbf{CC}$ commutes:
	\begin{equation*}\label{}
		\xymatrixcolsep{3pc}\xymatrixrowsep{4pc}\xymatrix
		{
			\cod \ar[rr]^{\iota(\cod)} \ar[d]_{m}
			&& \mathbb{L}(\sncd)_\mathfrak{R} \ar[d]^{\mathtt{S}\mathtt{C}(m)} \\
			\codp \ar[rr]^{\iota(\codp)} && \mathbb{L}(\sncdp)_\mathfrak{R}
		}
	\end{equation*}
	As a meticulous reader may see, the condition (\ref{eqnccmor}) is again crucial in this verification.	Hence the assignment $\cod\mapsto \iota(\cod)$ constitutes a natural isomorphism between the functors $1_{\mathbf{CC}}$ and $\mathtt{S}\mathtt{C}$.
\end{proof}

Combining the Lemmas \ref{lemcatiso1} and \ref{lemcatiso2}, we have the main theorem of this paper.
\begin{thm}\label{thmlrscc}
	The category $\mathbf{LRS}$ of left reductive regular semigroups is equivalent to the category $\mathbf{CC}$ of connected categories. 	
\end{thm}

\subsection{Right reductive regular semigroups}

We end this section with a brief discussion of the dual class of right reductive regular semigroups $\mathbf{RRS}$. In fact, much of the recent literature (see \cite{Billhardt2021,Ji2018,margolis2021} for instance) has focused on subclasses of $\mathbf{RRS}$ such as $\gr$-unipotent semigroups and left regular bands. Moreover, as our construction of $\mathbf{LRS}$ appears slightly non-symmetrical, one may wonder if the category $\mathbf{RRS}$ is \emph{dually equivalent} to $\mathbf{CC}$, this is not the case; we do obtain an equivalence of categories. To address all these, we provide a slightly more explicit discussion rather than merely stating the dual results.

Recall from Section \ref{secprel} (see equation (\ref{eqnrs})) that the category $\ros$ of principal right ideals of a regular semigroup $S$ is normal. Observe that given an arbitrary element $a\in S$, the principal cone $l^a$ is the cone in $\ros$ with vertex $fS$ given by, for each $eS\in v\ros$ 
$$l^a(eS):= l(e,ae,f) \text{ where } f\in E(R_a).$$

Also recall from \cite[Section 1.3]{clif} that the \emph{anti-regular representation} of $S$ is the anti-homomorphism $\lambda\colon S \to S_\lambda$ given by $a\mapsto \lambda_a$, and that $S$ is said to be \emph{right reductive} if $\lambda$ is injective. We may then show  the following dual statement of Theorem \ref{thmlslr}.

\begin{thm}\label{thmdualrs}
	Let $S$ be a regular semigroup. There is an anti-homomorphism $\bar{\lambda}\colon	S \to \srs$ given by $a\mapsto l^a$ and the semigroup $S$ is isomorphic to a subsemigroup of $\srso$ if and only if $S$ is right reductive.
\end{thm}

From Propositions \ref{proud2} and \ref{proud3}, given a connected category $\cod$, the connection  semigroup $\sncd$ is left reductive. Thus, the opposite semigroup $\sncdo:=(\sncd,\circ)$ where, for any cones $\gamma,\delta\in \sncd$,
\begin{equation}\label{eqnbinop}
	\gamma\circ\delta:=\delta\ast (\gamma(z_\delta))^\circ
\end{equation}
is right reductive and regular. In the sequel, we shall refer to this semigroup $\sncdo$ as the \emph{dual connection semigroup}. 

Now by \cite[Remark III.6]{cross}, we may see that for the  opposite semigroup $S^\text{op}$ of a regular semigroup $S$, we get
\begin{equation}
	\mathbb{R}(S^\text{op}) = \mathbb{L}(S)\: \text{ and }\: \mathbb{L}(S^\text{op})=\mathbb{R}(S).
\end{equation}
Moreover, the following relationships between its regular posets hold:
\begin{equation}
	S^\text{op}/\gr = S/\gl \: \text{ and } \: S^\text{op}/\gl=S/\gr.
\end{equation}
Therefore, $\mathbb{R}(\sncdo)=\mathbb{L}(\sncd)$ and using Proposition \ref{procatiso}, we get $\mathbb{L}(\sncd)$ isomorphic to $\mathcal{C}$. Hence we may state the next proposition.
\begin{pro}\label{procatisodual}
	Let $\cod$ be a connected category. The semigroup $\sncdo$ is a right reductive regular semigroup. The right category $\mathbb{R}(\sncdo)$ is a normal category isomorphic to the category $\mathcal{C}$ and the regular poset $\sncdo/\gl$ is isomorphic to $\mathfrak{D}$.
\end{pro}

At this point recall that by  Theorem \ref{thmdualrs},  given a right reductive regular semigroup $S$, we have $S\xhookrightarrow{} \srso$ and so, as in Section \ref{ssecslrs}, we may isolate the $\gl$-classes in $\srso$ of the form $\mathfrak{l}_a$.  Observe that $\mathfrak{l}_a$ is the $\gl$-class in $\srso$ containing the principal cone $l^a$, for $a\in S$. now, define 
$$\mathfrak{L}:= \{L_{l^e} : e \in E(S)\} =\{\mathfrak{l}_e : e\in E(S)\}.$$
Since $ \srso/\gl=\srs/\gr$, we have $\mathfrak{L}\subseteq \srs/\gr$ and  can verify the result below.

\begin{pro}
	Let $S$ be a right reductive regular semigroup.	The normal category $\mathbb{R}(S)$ is connected by the regular poset $\mathfrak{L}$, and so $\rl$ is a connected category. Moreover, the dual connection semigroup $\srlo$ is isomorphic to the semigroup $S$. 
\end{pro}

The next step is to extend this to a category equivalence on the category of right reductive regular semigroups $\mathbf{RRS}$. Given a semigroup homomorphism $\phi\colon S\to S'$ in $\mathbf{RRB}$, as in Lemma \ref{lemmorC}, we obtain a CC-morphism $m_\phi:=(F_\phi,G_\phi)$ from $\rl$ to $\rlp$. Notice  that $F_\phi$ remains a covariant functor. Hence, we obtain a functor $\mathtt{C}\colon\mathbf{RRS}\to\mathbf{CC}$ as follows:
$$S\mapsto\rl \text{ and }\phi\mapsto m_\phi.$$
Conversely, given a connected category $\cod\in v\mathbf{CC}$, the semigroup $\sncdo$ is a right reductive regular, and as in Proposition \ref{profunctorS}, we find a functor $\mathtt{S}\colon\mathbf{CC}\to \mathbf{RRS}$. Imitating the proofs of Lemmas \ref{lemcatiso1} and \ref{lemcatiso2}, we conclude the required equivalence.
\begin{thm}
	The category $\mathbf{RRS}$ of right reductive regular semigroups is equivalent to the category $\mathbf{CC}$ of connected categories. 	
\end{thm}

\section{$\gl$-unipotent semigroups}\label{seccxnlus}

The remainder of the paper is essentially dedicated to applications of the results in Section \ref{seccxnlrs} to various classes of left reductive regular semigroups. In this section, we  specialise the construction in the previous section to give an abstract construction of $\gl$-unipotent semigroups using \emph{supported} normal categories. These semigroups  were introduced and studied initially by Venkatesan \cite{venka1972,venka1974,venka1976b,venka1978} where he had called them \emph{right inverse semigroups}. Over the years, various facets of this class of semigroups (and of its dual, the class of $\gr$-unipotent semigroups) have been studied by many people including the second author \cite{Branco2010,Ji2018,Anthony2004,Edwards1979,Edwards1979a,Szendrei1985,Billhardt2021,Branco2009,Anthony2002,Gomes1985,Gomes1986,Gomes1986a}. We begin by recalling some basic properties of $\gl$-unipotent semigroups.  

\begin{dfn}\label{dfnlu}
A regular semigroup is said to be \emph{$\gl$-unipotent} if each $\gl$-class contains a unique idempotent.
\end{dfn}

\begin{pro}[{\cite[Theorem 1]{venka1974}\cite[Corollary 1.2, 1.3]{Edwards1979a}}]\label{prolus}
	Let $S$ be a regular semigroup. The following statements are equivalent:
	\begin{enumerate}
		\item $S$ is an $\gl$-unipotent semigroup;	
		\item $eS \cap fS = efS = feS$, for all $e,f\in E(S)$;	
		\item $efe=fe$, for all $e,f\in E(S)$, i.e., $E(S)$ is a right regular band;
		\item $a'a=a''a$, for all $a\in S$ and $a',a''\in V(a)$;
		\item the unique idempotent in an $\gl$-class $L_a$ containing $a$ is $a'a$, for all  $a'\in V(a)$;
		\item $a'ea=a''ea$, for all $a\in S$ and $a',a''\in V(a)$;
		\item $aa'ea=ea$, for all $a\in S$, $a'\in V(a)$ and $e\in E(S)$.
	\end{enumerate}
\end{pro}

By \cite[Theorem 4 (1)]{venka1974}, it is known that an $\gl$-unipotent semigroup is left reductive regular, i.e., the regular representation $a\mapsto\rho_a$ is injective. Since this fact is a cornerstone of this section, we record this formally amongst some other useful results regarding $\gl$-unipotent semigroups and provide a more transparent proof for the left-reductivity. In the process, we also characterise the natural partial order on an  $\gl$-unipotent semigroup. We begin with a useful observation.

\begin{lem}\label{lemlupo}
	Let $e,f $ be idempotents of an $\gl$-unipotent semigroup then $e \lel f$ if and only if $e\leqslant f$.
\end{lem}
\begin{proof}
	Suppose  $e \lel f$ in an $\gl$-unipotent semigroup so that $e=ef$. Then we have $e~\gl~fe~\leqslant~f$ in the semigroup. But by Definition \ref{dfnlu}, there is a unique idempotent in every $\gl$-class of an $\gl$-unipotent semigroup and so $e=fe$, i.e., $e\ler f$. Hence we have $e\leqslant f$. The converse is clear.   
\end{proof}
Recall that a regular semigroup whose idempotents form a band is said to be  \emph{orthodox}, and this is the case of any $\gl$-unipotent semigroup. In any orthodox semigroup $S$, for all $a\in S$, $a'\in V(a)$ and $e\in E(S)$, the elements $a'ea$ and $aea'$ are idempotents (see \cite[Proposition 6.2.2]{howie}).

On another hand, in any regular semigroup $S$, the natural partial order \cite{kss1980} is given by, for all $a, b \in S$,
$$ a\leqslant b \iff \exists {e,f\in E(S)},\:  a=be=fb.$$

When $S$ is $\gl$-unipotent, for all $a, b \in S$,
\begin{equation}\label{eqnlunpo}
	a\leqslant b \iff \exists f\in E(S), \:  a=fb.
\end{equation}

\begin{pro} \label{proluslrs}
	Let $S$ be an $\gl$-unipotent semigroup. For every pair of distinct elements $a,b$ in $S$, there exists an idempotent $f$ in $S$ such that $fa\ne fb$. In particular, the semigroup $S$ is left reductive. 	
\end{pro}
\begin{proof}
	We prove by contradiction. Suppose $fa=fb$ for every idempotent $f$. Since $aa'$ is an idempotent, we have $a=(aa')a=(aa')b$; similarly we get $b=(bb')b=(bb')a$. Using (\ref{eqnlunpo}), this implies $a\leqslant b$ and $b \leqslant a$; hence $a=b$, a contradiction. This concludes the first part of the proposition.
	
	Given $a\ne b$ in $S$, by the first part of the proposition we have an idempotent $f$ in $S$ such that $f\rho_a\ne f\rho_b$. Hence, the map $\rho\colon a\mapsto \rho_a$ is injective, whence  $S$ is left reductive.   
\end{proof}

Next, given an $\gl$-unipotent semigroup $S$, we shift our attention to the normal category $\los$. The following lemma is a simple consequence of Definition \ref{dfnlu}.
\begin{lem}\label{lemlosmor}
	In the category $\los$, given  morphisms $r(e,u,f)$ and  $r(g,v,h)$, we have $ r(e,u,f)= r(g,v,h)$ if and only if $e=g$, $f=h$ and $v=u$. Also, if   $ r(e,e,f)$ is an inclusion in $\los$,  the corresponding {\em unique} retraction is $ r(f,e,e)$.
\end{lem}  

Recall that in any regular semigroup $S$, the $\gr$-classes of $S$ form a  regular poset $(S/\gr,\sqsubseteq)$ as described in (\ref{eqnsr}). If $S$ is, in addition $\gl$-unipotent, then by Proposition \ref{prolus} (2), we have $eS\cap fS= efS= feS$, and so the regular poset becomes a semilattice $(S/\gr,\wedge)$ where the meet operation is given by
$$R_e\wedge R_f= R_{ef}.$$

The statement  (2) in the Proposition \ref{prolus}  may make one wonder if a regular semigroup $S$ such that its poset $S/\mathscr{R}$ is a semilattice is always an $\gl$-unipotent semigroup (also see (\ref{eqnsr})). This need not be the case as the following simple example shows. Consider a three element semigroup $T$ given by the $\gd$-class picture (on the left) in Figure \ref{figexsem}.
\begin{figure}[h]
	\begin{center}
		\begin{tikzpicture}[scale=.3]
			\foreach \y in {0,2,4} \foreach \x in {0,2} {\draw (0,\y)--(2,\y); \draw(\x,0)--(\x,4);}
			\fill[fill=gray,fill opacity=.25] (0,0) rectangle (2,4);
			
			\node () at (1,3) { $e$};
			\node () at (1,1) { $f$};
			
			\foreach \y in {-3,-1} \foreach \x in {0,2} {\draw (0,\y)--(2,\y); \draw(\x,-3)--(\x,-1);}
			\fill[fill=gray,fill opacity=.25] (0,-3) rectangle (2,-1);
			\node () at (1,-2) { $0$};
			
			\node () at (10,2) { $R_e$};
			\node () at (14,2) { $R_f$};
			\node at (12,-2) {$R_0$};
			
			\draw  (10,1.3) to  (11.8,-1.3) ;
			
			\draw  (14,1.3) to  (12.2,-1.3) ;
			
		\end{tikzpicture}
	\end{center}
	\caption{A regular semigroup $T$ such that its poset $T/\mathscr{R}$ is a semilattice but $T$ is not $\gl$-unipotent.}
	\label{figexsem}	
\end{figure}	

This semigroup  $T$ is clearly not $\gl$-unipotent (as $e$ and $f$ are distinct $\gl$-related idempotents) although the regular poset $T/\gr$  forms a three element semilattice (given on the right side of Figure \ref{figexsem}). Observe that $eT\cap fT=0T=\{0\}$.

Notice that the (dual) semilattice of $\gl$-classes which appears in the context of the left regular bands has been referred to as the \emph{support semilattice} of the semigroup \cite{margolis2021, brown2000}. This underlying semilattice plays a crucial role in the structure of these semigroups. We shall be discussing the case of right and left regular bands in Section \ref{seccxnrrb}.  With this terminology in mind, we proceed to give the construction of an $\gl$-unipotent semigroup by introducing \emph{supported} categories, as specialisations of connected categories.

\begin{dfn}\label{dfnsupc}
	 A connected category $\cod$ is said to be \emph{supported} if each $c\in v\mathcal{C}$ is connected by a \emph{unique} $\mathfrak{d}\in \mathfrak{D}$.	
\end{dfn}

\begin{pro}\label{prosncsup}
	Let $\cod$ be a supported category. Then the connection semigroup $\sncd$ is $\gl$-unipotent.	
\end{pro}
\begin{proof}
	Given a supported category $\cod$, it is connected and so by Proposition \ref{proud2}, we know that $\sncd$ is a regular semigroup. Now, using Lemmas \ref{lemgrcd} and \ref{lemidgrcd}, given a cone $\gamma\in\sncd$ with vertex $c$, we can see that any idempotent in the $\gl$-class of $\gamma$ is of the form $\epsilon(c,\mathfrak{d})$, for some $\mathfrak{d}\in \mathfrak{D}$ connecting $c$. However, since $\cod$ is a supported category, there is a unique $\mathfrak{d}$ with this property. Hence each $\gl$-class of $\sncd$ contains a unique idempotent $\epsilon(c,\mathfrak{d})$ and by Definition \ref{dfnlu}, the semigroup $\sncd$ is  $\gl$-unipotent. 
\end{proof}

\begin{lem}\label{lemsl}
Given a supported category $\cod$,  the down-set $\mathfrak{D}$ is  a subsemilattice of $\sncd/\gr$.
\end{lem}
\begin{proof}
	Since a supported category $\cod$ is connected, using Proposition \ref{procatiso}, the down-set $\mathfrak{D}$ is order isomorphic to the poset $\sncd/\gr$. But by Proposition \ref{prosncsup},  the semigroup $\sncd$ is $\gl$-unipotent and so using Proposition \ref{prolus} (2), we see that the $\sncd/\gr$ is a semilattice. Hence the down-set $\mathfrak{D}$ of a supported category is  a semilattice.
\end{proof}

\begin{rmk}\label{rmksupmap}
	If $\cod$ is a supported category, by Definition \ref{dfnsupc}, there is a well-defined mapping $\Gamma:v\mathcal{C}\to\mathfrak{D}$. This is a surjection, and in the sequel it will be referred to as the \emph{support map}. Observe that the support  map $\Gamma$ need not be injective in general, but we shall later see that $\Gamma$ is always order preserving (see Proposition \ref{proidlus}). Also note that in contrast to the support map of \cite{margolis2021}, our map $\Gamma$ is not a homomorphism as there is no semigroup structure in the set $v\mathcal{C}$. 
\end{rmk}

Given a supported category $\cod$, we shall say that the category $\mathcal{C}$ is \emph{supported by the semilattice} $\mathfrak{D}$.    We now record the following specialisations of Lemmas \ref{lemgrcd} and  \ref{lemidgrcd} in the case of supported categories. 
\begin{pro}\label{proidlus}
	Given a supported category $\cod$, let $\epsilon_1=\epsilon(c_1,\mathfrak{d}_1)$ and $\epsilon_2=\epsilon(c_2,\mathfrak{d}_2)$ be idempotents in the semigroup $\sncd$, then:
	\begin{enumerate}
		\item $\epsilon_1=\epsilon_2\iff c_1= c_2$, and there is a bijection between the sets $E(\sncd)$ and $v\mathcal{C}$;
		\item $\epsilon_1\ler\epsilon_2\iff \mathfrak{d}_1=\mathfrak{d}_1\wedge \mathfrak{d}_2$.
	\end{enumerate} 
	In particular, the support map $\Gamma\colon v\mathcal{C}\to \mathfrak{D}$ defined by $ c\mapsto R_{\epsilon(c,\mathfrak{d})}$ is an order preserving surjection.
\end{pro}
\begin{proof}
	(1) follows from the fact that there is a unique idempotent in the semigroup $\sncd$ with a given vertex. 
	To prove  (2), first observe that by Lemma \ref{lemsl}, the down-set $\mathfrak{D}$ is a semilattice and in $\mathfrak{D}$, we have $\mathfrak{d}_1\sqsubseteq \mathfrak{d}_2$ if and only if $\mathfrak{d}_1=\mathfrak{d}_1\wedge \mathfrak{d}_2$. Also,  by Proposition \ref{procatiso}, we know that $\sncd/\gr$ is order isomorphic to $\mathfrak{D}$. So,  $$\epsilon_1\ler\epsilon_2\iff \epsilon_1\sncd\sqsubseteq \epsilon_2\sncd \iff \mathfrak{d}_1\sqsubseteq \mathfrak{d}_2 \iff \mathfrak{d}_1=\mathfrak{d}_1\wedge \mathfrak{d}_2.$$
	To show the last part of this proposition, first observe that the map $\Gamma$ is well-defined by (1) and by definition, it is a surjection. Now if $c_1\preceq c_2$, then by Lemma \ref{lemidgrcd} (1), we get $\epsilon_1 \lel \epsilon_2$. Since $\sncd$ is $\gl$-unipotent, using Lemma \ref{lemlupo}, we have $\epsilon_1\ler \epsilon_2$. So by  Lemma \ref{lemidgrcd} (2), we obtain $\mathfrak{d}_1\sqsubseteq \mathfrak{d}_2$ and therefore $\Gamma$ is order preserving.
\end{proof}

We have seen above how a supported category gives rise to an $\gl$-unipotent semigroup. The proposition below shows the converse, i.e. every $\gl$-unipotent semigroup determines a  supported category.

\begin{pro}\label{proluslr}
	Let $S$ be 	an $\gl$-unipotent semigroup $S$. Then $\lr$ is a supported category. The support map $\Gamma_S\colon v\los\to\mathfrak{R}$ is given by $Se\mapsto\mathfrak{r}_e$.
\end{pro} 
\begin{proof}
	Given an $\gl$-unipotent semigroup $S$, by Proposition \ref{proluslrs}, it is left reductive.  As in Section \ref{ssecslrs}, we can show that the category $\los$ is normal and the semigroup $\sls$ is regular. Now,  we define the down-set $\mathfrak{R}\subseteq\sls/\gr$ as:
	$$\mathfrak{R}:= \{\mathfrak{r}_e : e\in E(S) \}.$$
	We know that $S/\gr$ is order isomorphic to $\mathfrak{R}$ via $R_e\mapsto\mathfrak{r}_e:=R_{r^e}$, and in addition it is a meet semilattice when $S$ is $\gl$-unipotent. Hence $(\mathfrak{R},\wedge)$ is a meet semilattice with
	$$\mathfrak{r}_e\wedge \mathfrak{r}_f= \mathfrak{r}_{ef}.$$
	Also, since each $\gl$-class in $S$ contains a unique idempotent, each object $Se$ in $v\los$ is connected by a unique $\mathfrak{r}_e\in \mathfrak{R}$. Further, using Lemma \ref{lemlupo}, we have $e \lel f$ if and only if $e\leqslant f$. So,
	$$Se\preceq Sf\iff e\lel f\iff e\leqslant f  \implies e \ler f \iff R_e\sqsubseteq R_f \iff \mathfrak{r}_e \sqsubseteq \mathfrak{r}_f.$$
	Hence the support map $\Gamma_S\colon Se \mapsto \mathfrak{r}_e$ is an order preserving surjection from $v\los$ to $\mathfrak{R}$.
\end{proof}

Specialising Proposition \ref{prolsriso}, we get:
\begin{pro}\label{prolusiso}
	Given an $\gl$-unipotent semigroup $S$, the connection semigroup $\slr$ is isomorphic to $S$.
\end{pro}

Further, the discussion in Section \ref{sseccateqlrs} carries over verbatim to the $\gl$-unipotent case. It is clear that $\gl$-unipotent semigroups form a full subcategory of $\mathbf{LRS}$, say $\mathbf{LUS}$ and supported categories form a full subcategory of $\mathbf{CC}$, say $\mathbf{SC}$. We shall refer to the morphisms in the subcategory $\mathbf{SC}$ as \emph{SC-morphisms}, in the sequel.

Repeating the arguments in Section \ref{sseccateqlrs} to obtain Theorem \ref {thmlrscc}, we can show the following.

\begin{thm}\label{thmcatlus}
	The category $\mathbf{LUS}$ of $\gl$-unipotent semigroups is equivalent to the category $\mathbf{SC}$ of supported categories.
\end{thm}  

The dual result for $\gr$-unipotent semigroups also holds. 
\begin{cor}
The category $\mathbf{RUS}$ of $\gr$-unipotent semigroups is equivalent to the category $\mathbf{SC}$ of supported categories.
\end{cor}

\section{Right regular bands}\label{seccxnrrb}

Now we further specialise the construction in Section \ref{seccxnlus} to describe right (and left) regular bands and, in this case, we obtain an adjunction. Recall that a \emph{band} is a semigroup such that every element is an idempotent. A \emph{right regular band} is a band which satisfies the identity $efe=fe$. Observe that the idempotents of an $\gl$-unipotent semigroup forms a right regular band (see  Proposition \ref{prolus} (3)). Hence, right regular bands are a special class of $\gl$-unipotent semigroups. We begin by the observation that right regular bands form a  full coreflective subcategory $\mathbf{RRB}$ of the category $\mathbf{LUS}$. 
To guide the readers to this end, we recall the following definitions which shall also be needed later in this section.
\begin{dfn}\label{dfnadjoint}
	Let $\mathbf{C}$ and $\mathbf{D}$ be two arbitrary categories. An \emph{adjunction } $\mathbf{C}\to\mathbf{D}$ is a triple $(\mathtt{F}, \mathtt{G}, \eta)$, where $\mathtt{F}\colon \mathbf{C} \to \mathbf{D}$ and $\mathtt{G}\colon\mathbf{D}\to \mathbf{C}$ are functors, and $\eta$ is a natural transformation $1_{\mathbf{C}} \to \mathtt{F}\mathtt{G}$ such that the following condition holds:
	
	\begin{itemize}
		\item 	For every pair of objects $\mathcal{C}\in v\mathbf{C}$ and $\mathcal{D} \in v\mathbf{D}$, and for every morphism $\phi \colon \mathcal{C}\to \mathtt{G}(\mathcal{D})$ in the category $\mathbf{C}$, there exists a unique morphism $\bar{\phi} \colon \mathtt{F}(\mathcal{C}) \to \mathcal{D}$ in the category $\mathbf{D}$ such that the following diagram commutes:
		\begin{equation*}\label{}
			\xymatrixcolsep{5pc}\xymatrixrowsep{4pc}\xymatrix
			{
				\mathcal{C} \ar[rd]^{\phi} \ar[d]^{\eta_\mathcal{C}}\\
				\mathtt{G}(\mathtt{F}(\mathcal{C}))\ar[r]^{\mathtt{G}(\bar{\phi})} &\mathtt{G}(\mathcal{D})		
			}
		\end{equation*}
	\end{itemize}	
	In this case, $\mathtt{F}$ and $\mathtt{G}$ are called \emph{left } and \emph{right adjoints }, respectively, and $\eta$ is the \emph{unit of adjunction}. 
\end{dfn}

We refer the reader to \cite[Proposition 1.3]{gabriel1967} for several equivalent characterisations of the Definition {\ref{dfnadjoint} above. 
	
	\begin{dfn}
	A \emph{coreflective} subcategory is a full subcategory $\mathbf{C}$ of a category $\mathbf{D}$ whose inclusion functor $\mathtt{J}\colon\mathbf{C}\to\mathbf{D}$ has a right adjoint. We shall say that a category $\mathbf{C}$ is \emph{coreflective} in $\mathbf{D}$ if $\mathbf{C}$ is equivalent to a coreflective subcategory of $\mathbf{D}$. 
\end{dfn}

Informally, a coreflective subcategory specifies a distinguished class of subobjects in the ambient category, and the \emph{coreflector functor} selects, for each object, its associated subobject. Now, it is routine to verify the following lemma where the {coreflector} $\mathtt{B}\colon \mathbf{LUS}\to \mathbf{RRB}$ maps an $\gl$-unipotent semigroup $S$ to its right regular band $E(S)$ of idempotents. 

\begin{lem}\label{lemrrb}
	The category $\mathbf{RRB}$ of right regular bands is a coreflective subcategory of the category $\mathbf{LUS}$ of $\gl$-unipotent semigroups.
\end{lem}

Combining Lemma \ref{lemrrb} and Theorem \ref{thmcatlus}, one can construct an adjunction between the categories $\mathbf{RRB}$ and $\mathbf{SC}$. But we shall briefly exposit this adjunction in a direct manner as it illustrates how we can build right regular bands from supported categories. Putting together  Proposition \ref{prosncsup} and Proposition \ref{prolus} (3), we are led to the next lemma.
\begin{lem}
	Let $\cod$ be a supported category. Then the set $E(\sncd)$ of all idempotent cones in the connection semigroup $\sncd$
	$$	E(\sncd) :=\{\: \epsilon(c,\mathfrak{d}) : c \text{ is connected by } \mathfrak{d} \: \}$$
	forms a right regular band. 
\end{lem}

Notice that by Proposition \ref{proidlus} (1), the band $E(\sncd)$ is in bijection with the set $v\mathcal{C}$.
Further, given an SC-morphism $m:=(F,G)$ from $\cod$ to $\codp$, as shown in Lemmas \ref{lemphi1} and \ref{lemphi2}, we can prove that $\phi_m\colon E(\cod) \to E(\codp)$ given by:
$$\phi_m\colon \epsilon(c,\mathfrak{d}) \mapsto \epsilon (F(c),G(\mathfrak{d})) $$
is a semigroup homomorphism of right regular bands. Hence we obtain a functor $\mathtt{E}\colon \mathbf{SC} \to \mathbf{RRB}$ as follows:
$$\cod\mapsto E(\sncd)\: \text{ and } \: m \mapsto \phi_m. $$

Conversely starting from a right regular band $S$, since it is also $\gl$-unipotent, we can easily see that (as shown in Section \ref{seccxnlus}) the category 
$\lr$ constitutes a supported category. This correspondence $S\mapsto \lr$ is given by the functor $\mathtt{C}\colon\mathbf{RRB}\to\mathbf{SC}$. The functor $\mathtt{C}$ is precisely the restriction of the functor defined in the Section \ref{sseccateqlrs} to the category $\mathbf{RRB}\subseteq\mathbf{LUS}\subseteq\mathbf{LRS}$.  Further, we will show that the functor $\mathtt{C}$ is a left adjoint to the functor $\mathtt{E}$. 

\begin{thm}
	There is an adjunction from the category $\mathbf{RRB}$ of right regular bands to the category $\mathbf{SC}$ of supported categories. In particular, the category $\mathbf{RRB}$ is coreflective in the category $\mathbf{SC}$. 
\end{thm}
\begin{proof}
	To begin with, observe that given a right regular band $B$, it is left reductive and by Proposition \ref{prolsriso}, the connection semigroup $\widehat{\lrb}$ is isomorphic to $B$ via the map $\bar{\rho}(B)$. Moreover, $E(\widehat{\lrb})=\widehat{\lrb}$. Hence the assignment $B\mapsto\bar{\rho}(B)$ is a natural isomorphism from the functor $1_\mathbf{RRB}$ to the functor $\mathtt{C}\mathtt{E}$.
	
	Now, given an object $B\in v\mathbf{RRB}$, by Proposition \ref{proluslr}, we can see that $\mathtt{C}(B)=\lrb$ is a supported category. Let $\cod\in v\mathbf{SC}$ so that $\mathtt{E}({\cod})=E(\sncd)$. Given  a semigroup homomorphism $\phi\colon B\to \mathtt{E}({\cod})$ in $\mathbf{RRB}$, by Lemma \ref{lemmorC} and Proposition \ref{procatiso}, we see that $m_\phi:=(F_\phi,G_\phi)$ is the unique SC-morphism from $\lrb$ to $\codp\subseteq\cod$, where $\codp:=\mathtt{C}(\mathtt{E}(\widehat{\cod}))$. Next, we may routinely verify that the following diagram commutes:
	\begin{equation*}\label{}
		\xymatrixcolsep{5pc}\xymatrixrowsep{4pc}\xymatrix
		{
			B \ar[rd]^{\phi} \ar[d]^{\bar{\rho}(B)}\\
			\mathtt{E}(\mathtt{C}(B))\ar[r]^{\mathtt{E}(m_\phi)} &\mathtt{E}(\cod)		
		}
	\end{equation*}
	Hence $(\mathtt{C},\mathtt{E},\bar{\rho})$ constitutes an adjunction from the category $\mathbf{RRB}$ to $\mathbf{SC}$. The last part of the theorem follows directly from the fact that  the left adjoint $\mathtt{C}$ is fully-faithful (see \cite[Proposition 1.3]{gabriel1967}). To conclude one can check that $\mathbf{RRB}$ is equivalent to the category $\mathtt{C}(\mathbf{RRB})$, and the latter is a coreflective subcategory of $\mathbf{SC}$. 	
\end{proof}

Dually, we have the following corollary for the more `popular' class of left regular bands:

\begin{cor}
The category $\mathbf{LRB}$ of left regular bands is coreflective in the category $\mathbf{SC}$ of supported categories. 
\end{cor}

\section{Inverse semigroups}\label{secinvs}
In this section, we look at a class of regular semigroups, which are both left and right reductive, namely inverse semigroups. Inverse semigroups arguably form the most important class of regular semigroups, mainly due to their ability to capture partial symmetry \cite{lawson}.
\begin{dfn}\label{dfninv}
	A regular semigroup is called \emph{inverse} if all of its idempotents commute.
\end{dfn}

In addition to specialising our construction, Theorem \ref{thminvssc} below may be realised as a weaker version of the ESN Theorem. ESN Theorem exposits a \emph{category isomorphism} between the category of inverse semigroups and the category of \emph{inductive groupoids} \cite[Section 4.1]{lawson}.  In our construction, an object does not remain the same when passing from the category of semigroups to the category of categories, in contrast with the situation in ESN Theorem. As mentioned in Section \ref{secintro}, we represent an element in the semigroup  as a cone in the category so that the original semigroup is only \emph{isomorphic} to the semigroup of cones. As a result, our correspondence will never be a category isomorphism, even for the group case.

In the joint work \cite{locinverse}, the first author described a category equivalence between inverse semigroups and inversive categories. That construction used Nambooripad's normal categories and admittedly, the description did not reflect the symmetrical nature of inverse semigroups. In contrast, when we employ supported categories, the symmetry of the semigroups gets manifested by the categories `supporting' themselves; so we dub these \emph{self-supported categories}. 

Before continuing we recall some characterisations of inverse semigroups:
\begin{pro}[{\cite[Theorem II.2.6]{grillet}}]\label{proinv}
 The following are equivalent:
	\begin{enumerate}
		\item $S$ is an inverse semigroup;
		\item every element in $S$ has a unique inverse element;
		\item $(E(S),\leqslant)$ is a semilattice;
		\item there is a unique idempotent in each $\mathscr{L}$-class and each  $\mathscr{R}$-class of $S$.
	\end{enumerate}
\end{pro}

By Proposition \ref{proinv} (4), inverse semigroups are $\gl$-unipotent semigroups, which are left-right symmetrical. In fact, they may be regarded as the most `symmetric' class of semigroups after groups. This symmetry is a reflection of the uniqueness of the inverse, which in turn, defines a natural involution on the semigroup given  by $a\mapsto a^{-1}$. This leads to the proposition below.

\begin{pro}
	Let $S$ be an inverse semigroup. Then the left category $\los$ is normal category isomorphic to the right category $\ros$. In particular, the semilattice $(v\los,\preceq)$ is order isomorphic to $(v\ros,\sqsubseteq)$.
\end{pro}
\begin{proof}
	Define a functor $F\colon \los\to \ros$ by, for any inverse semigroups $S$,
	\begin{equation}
		vF(Se):= eS \text{ and } F(r(e,u,f)):= l(e,u^{-1},f).
	\end{equation}
	Since inverse semigroups have a unique idempotent in every $\gl$-class and in every $\gr$- class, it is easy to see that the map $vF$ is a well-defined bijection. Now, by Proposition \ref{proinv} (3), the quasi-orders $\lel$ and $\ler$ on the idempotents of an inverse semigroups $S$ coincide with the natural partial order $\leqslant$,  and so
	$$Se\preceq Sf \iff e \lel f \iff e \leqslant f \iff e\ler f \iff eS \sqsubseteq fS. $$
	Hence $vF$ is an order isomorphism between the semilattices $v\los$ and $v\ros$. Also, observe that $v\los$ is order isomorphic to the set $E(S)$ of idempotents of $S$.
	
	Given $u\in eSf$ such that $r(e,u,f)$ is a morphism in $\los$ from $Se$ to $Sf$, we can see that $u^{-1} \in f^{-1}Se^{-1}=fSe$ so that $l(e,u^{-1},f)$ is a morphism in $\ros$ from $eS$ to $fS$. Then using Lemma \ref{lemlosmor} and Proposition \ref{proinv}  (2), we  verify that the map $F$ is well defined. Now given two composable morphism $r(e,u,f)$ and $r(g,v,h)$ in the category $\los$ such that $f\gl g$, we know that $f=g$ and  
	$$F(r(e,u,f)r(g,v,h))= F(r(e,uv,h))=l(e,{(uv)}^{-1},h)=l(e,v^{-1}u^{-1},h).$$
	On the other hand,
	$$F(r(e,u,f))F(r(g,v,h))=l(e,u^{-1},f)l(g,v^{-1},h)=l(e,v^{-1}u^{-1},h).$$
	Hence $F$ is a covariant functor from $\los$ to $\ros$. Using Lemma \ref{lemlosmor}, one sees that $F$ is inclusion preserving and fully-faithful. Therefore, $F$ is a normal category isomorphism.
\end{proof}

From Proposition \ref{proluslr}, the support map $\Gamma\colon v\los\to \mathfrak{R}$ of an $\gl$-unipotent semigroup $S$ is given by $Se\mapsto \mathfrak{r}_e$. In the case of inverse semigroups,  the next corollary reflects the left-right symmetry of these semigroups. 

\begin{cor}\label{corinv}
	Let $S$ be an inverse semigroup. The support map $\Gamma\colon v\los\to \mathfrak{R}$ is an order isomorphism.
\end{cor}

\begin{proof}
	The map $vF$ in the previous proposition may be interpreted as an order isomorphism from $v\los$ to $S/\gr$ given by $Se\mapsto R_e$. As discussed in Section \ref{ssecslrs}, we know that the map $R_e\mapsto \mathfrak{r}_e$ is an order isomorphism from the poset $S/\gr$ to $\mathfrak{R}$. Hence $v\los$ is order isomorphic to $\mathfrak{R}$ via the map $Se\mapsto\mathfrak{r}_e$.
\end{proof}

\begin{rmk}
	
	When $S$ is an inverse semigroup, we see that $v\los$ is order isomorphic to $\mathfrak{R}$. In other words, the normal category $\los$ is supported by a poset which is isomorphic to $v\los$. Or by abuse of terminology, we can just say that the normal category $\los$ is supported by $v\los$, i.e., $\los$ is self-supported.
\end{rmk}

\begin{dfn}
	A supported category is said to be \emph{self-supported} if the support map is an order isomorphism.
\end{dfn}

\begin{rmk}
	In the above definition, we do not need to explicitly specify the supporting semilattice $\mathfrak{D}$ as we know that $\mathfrak{D}$ is the semilattice $\Gamma(v\mathcal{C})$. 	
\end{rmk} 

\begin{pro}
	Given a self-supported category $\mathcal{C}$ with $\mathfrak{D}:=\Gamma(v\mathcal{C})$, the connection semigroup $\sncd$ is  inverse. 
\end{pro} 
\begin{proof}
	Since a self-supported category is supported, by Proposition \ref{prosncsup} the connection semigroup $\sncd$ is $\gl$-unipotent. Now let $\epsilon_1=\epsilon(c_1,\mathfrak{d})$ and $\epsilon_2=\epsilon(c_2,\mathfrak{d})$ be $\gr$-related idempotents of $\sncd$. Then the objects $c_1$ and $c_2$ are both connected by $\mathfrak{d}$. As the support map $\Gamma\colon v\mathcal{C}\to \mathfrak{D}$ is injective, we get  $c_1=c_2$. Hence $\epsilon_1=\epsilon_2$ and each $\gr$-class in $\sncd$ contains a unique idempotent. By Proposition \ref{proinv}  (4), the semigroup $\sncd$ is inverse.
\end{proof} 

Using Corollary \ref{corinv} and specialising Proposition \ref{proluslr}, we conclude the next statement.

\begin{pro}
	Let $S$ be an inverse semigroup. Then the category $\los$ is self-supported  such that $\mathfrak{R}:=\Gamma_S(\los)$ is order isomorphic to $v\los$.
\end{pro}

Finally,  Theorem \ref{thmcatlus} allows us to obtain the following equivalence theorem for inverse semigroups which may be realised as a weaker form of ESN Theorem.

\begin{thm}\label{thminvssc}
	The category of inverse semigroups is equivalent to the category of self-supported categories.
\end{thm}

\section{Totally left reductive semigroups and regular monoids}\label{secrm}
In this section we aim to study regular monoids. Observe that they are both left and right reductive. As a precursor to regular monoids, we identify a new class of left reductive regular  semigroups, which we call \emph{totally left reductive}. In the process, we realise a couple of interesting results regarding the semigroup of all cones coming from an arbitrary normal category $\mathcal{C}$ (also see Corollary \ref{cortlclr}). We introduce morphisms between arbitrary normal categories and thereby define the category of normal categories. The discussion in this section tells us that if a regular semigroup is totally left reductive (in particular, if it is a monoid), the cross-connection analysis is expendable and normal categories suffice to describe  such semigroups completely.  

\subsection{Totally left reductive semigroups}\label{sstlrs}

Recall that, given a regular semigroup $S$, we have the  homomorphism $\bar{\rho}\colon S \to \sls$,  $a\mapsto r^a$, and by Theorem \ref{thmlslr}, this is injective when $S$ is left reductive. The question of the surjectivity of $\bar{\rho}$ leads to the following definition.
\begin{dfn}
	A regular semigroup $S$ is said to be \emph{totally left reductive} if $S$ is isomorphic to the semigroup $\sls$ of cones in the category $\los$.
\end{dfn}

It is not difficult to see that regular monoids are totally left reductive semigroups (see Proposition \ref{prorm}), but this class also contains several non-monoids. It has been showed in \cite[Theorem 3.1]{tx}  and  in \cite[Theorem 2]{tlx}, respectively, that singular transformation semigroups and singular linear transformation semigroups are totally left reductive. The semigroups of order preserving mappings on finite chains \cite{sneha2024} and the Clifford inverse semigroups \cite{azpr2023} also fall into this class. All the above mentioned papers were written within the framework of Nambooripad's cross-connection theory. Our next aim is to employ connected categories to look at these classes, and, as the reader may see, our analysis will take us to easier and stronger characterisations of each of these classes. 

Recall from Section \ref{seccxnlrs} that we can treat an arbitrary normal category $\mathcal{C}$ as a connected category $\cod$ by taking $\mathfrak{D}= \snc /\gr$. For the remaining part of the paper, we shall  treat normal categories in this manner without further comment. Also remember by Corollary \ref{cortlclr} that the semigroup $ \snc $ of all cones in a normal category $\mathcal{C}$ is  left reductive such that $\mathbb{L}(\snc)\cong \mathcal{C}$. The next lemma follows by a routine verification.
\begin{lem}\label{lemiso}
	Let $\mathcal{C}$ and $\mathcal{C'}$ be isomorphic normal categories. Then the regular semigroup $\snc$ is isomorphic to the semigroup $\widehat{\mathcal{C'}}$.	
\end{lem}

Further, combining Corollary \ref{cortlclr} and Lemma \ref{lemiso}, we conclude that the semigroup  $\snc$  is isomorphic to the semigroup $\widehat{\mathbb{L}(\snc)}$ of cones in the left category $\mathbb{L}(\snc)$. Hence, we get the next statement.

\begin{pro}\label{pronctlrs}
	Let $\mathcal{C}$ be a normal category. The semigroup $\snc$ of all cones in $\mathcal{C}$ is a totally left reductive  semigroup.
\end{pro}

Summarising the above discussion: a normal category $\mathcal{C}$ can be realised as a connected category $\mathcal{C}_{\snc/\gr}$ such that the connection semigroup $\snc=\widehat{\mathcal{C}_{\snc/\gr}}$ is  totally left reductive,  and the left category $\mathbb{L}(\snc)$ is isomorphic to $\mathcal{C}$. Conversely, given a totally left reductive semigroup $S$, the category $\los$ is normal, and the semigroup $\sls$ is isomorphic to $S$. Thus, the regular poset $\sls/\gr$ is isomorphic to $S/\gr$, whence the normal category $\los$ may be realised as the connected category $\los_{S/\gr}$ and the connection semigroup $\widehat{\los_{S/\gr}}=\sls$ is isomorphic to $S$.

Now, normal categories (with CC-morphisms) form a full subcategory of the category $\mathbf{CC}$, and totally left reductive semigroups (with regular semigroup homomorphisms) form a full subcategory of the category $\mathbf{LRS}$. Therefore, specialising Theorem \ref{thmlrscc}, we obtain:

\begin{thm}\label{thmnccat}
	The category of normal categories is equivalent to the category of totally left reductive semigroups.
\end{thm}

\subsection{Regular monoids}\label{ssecrm}

Now we look at the most important class of totally left reductive semigroups, namely that of regular monoids. Although the construction we present in this special case does not seem very insightful regarding the ideal structure of the monoids (as a monoid itself forms a two sided ideal, and  consequently, the left and the right actions on this ideal determine the entire actions), our analysis does provide a category equivalence of regular monoids with small categories which  we believe may prove to be quite useful.

To begin with, let $S$ be a regular monoid with identity $1$. The left ideal category $\los$ of  $S$ has a largest object, namely $S1=S$. Hence we have the corollary to Theorem \ref{thmlslr} below. 

\begin{pro}[{\cite[Corollary III.17]{cross}}]\label{prorm}
Let $S$ be a regular monoid. The normal category $\los$ has a largest object, and $S$ is isomorphic to the semigroup $\sls$ of all cones in $\los$.
\end{pro}

The above proposition leads us to identify certain special normal categories.
\begin{dfn}
A normal category $\mathcal{C}$ is said to be \emph{bounded above} if there exists an object $\mathfrak{d}\in v\mathcal{C}$ such that $c\preceq d$, for every $c\in v\mathcal{C}$.
\end{dfn}

\begin{lem}
Let $\mathcal{C}$ be a normal category which is bounded above. Then the semigroup $\snc$ of all cones in $\mathcal{C}$ is a regular monoid.
\end{lem} 

\begin{proof}
By Lemma \ref{lemrs}, we know that $\snc$ is a regular semigroup. Since $\mathcal{C}$ is bounded above, it has a largest object, say $k$. Let $\epsilon_k$ be an  idempotent cone in $\mathcal{C}$ with vertex $k$. Now, for $c\in v\mathcal{C}$, we have $c\preceq k$, and  given an arbitrary cone $\gamma\in \snc$,  by Definition \ref{dfnnormcone} (1) we get $\gamma(c)=j(c,k)\gamma(k)$. In particular, $\epsilon_k(c)=j(c,k)\epsilon_k(k)=j(c,k)1_k=j(c,k)$. Observe that $\gamma(k)$ will always be an epimorphism. Then using equation (\ref{eqnbin}), we see that
$$\gamma\:\epsilon_k =\gamma\ast (\epsilon_k(z_\gamma))^\circ=\gamma\ast (j(z_\gamma,k))^\circ=\gamma\ast 1_{z_\gamma}=\gamma.$$ 
Also for an arbitrary $c\in v\mathcal{C}$,
$$\epsilon_k \gamma (c)= \epsilon_k(c) (\gamma(k))^\circ= j(c,k) \gamma(k)=\gamma(c).$$
Hence $\epsilon_k \gamma =\gamma= \gamma \epsilon_k$, for any $\gamma\in \snc$. Thus the semigroup $\snc$ is a monoid with identity $\epsilon_k$.
\end{proof}

Since regular monoids form a full subcategory of the category of totally left reductive semigroups, and bounded above categories form a full subcategory of normal categories, specialising Theorem \ref{thmnccat}, we get:
\begin{thm}\label{thmbanc}
The category of regular monoids is equivalent to the category of bounded above normal categories (with CC-morphisms).
\end{thm}

\subsection{Transformation semigroups}\label{ssectrans}

The full transformation monoid $\tn$ is  regular  and it contains the symmetric group $\mathscr{S}_n$ as its subgroup of units. From Section \ref{ssfpn},  the full powerset category $\fpn_\Pi$ of all subsets of $\ns$ is  connected and normal with largest object $\ns$,  hence $\fpn_\Pi$ is bounded above. We denote  $\fpn_\Pi$ simply  by $\fpn$. 

In the light of Theorem \ref{thmbanc}, the next proposition leads to a full description of the regular monoids $\tn$, in terms of categories.

\begin{pro}\label{proltn}
The normal category $\mathbb{L}(\tn)$ is isomorphic to the full powerset category $\fpn$, as bounded above categories.
\end{pro}

\begin{proof}
From Lemma \ref{lemgrtn} (1)  the left ideals in the monoid $\tn$ are determined by images. Hence, given  $\alpha,\beta\in E(\tn)$ and $\theta\in \alpha\tn\beta$ , define a functor $F\colon \mathbb{L}(\tn)\to \fpn$ as:
$$vF(\tn \alpha):= \ns\alpha \text{ and } F(r(\alpha,\theta,\beta)):= \theta_{|{\ns\alpha}}.$$
It may be routinely verified that $F$ is a normal category isomorphism from $\mathbb{L}(\tn)$ to $\fpn$. 

Notice that the largest object in $\mathbb{L}(\tn)$ is $\tn$ and so, we can define a map $G\colon\widehat{\mathbb{L}(\tn)}/\gr\to \Pi$ as $R_\gamma \mapsto \pi_{\gamma(\tn)}$. Since $\widehat{\mathbb{L}(\tn)}$ is isomorphic to $\tn$ and using Lemma \ref{lemgrtn}  (2), we may verify that $G$ is an order isomorphism. Further, we may prove that the pair $(F,G)$  also satisfies the condition (\ref{eqnccmor}), and hence the category $\mathbb{L}(\tn)$ is isomorphic to $\fpn$ as bounded above categories.
\end{proof}	

The above proposition tells us that the full transformation monoid $\tn$ is equivalent to the bounded above category $\fpn$. Hence, any question regarding the monoid $\tn$ may be translated to an equivalent question regarding the connected category $\fpn_\Pi$. In particular,  we can obtain the exact same descriptions of the biordered set and the sandwich sets of $\tn$ in terms of subsets and partitions (see \cite[Section 5.1]{var}) using the connected category description rather than cross-connections. Observe that the cross-connection description of $\tn$ previously known, \cite{tx}, involved the rather cumbersome category of partitions\cite{tx2}, but our approach bypasses this by just using the poset $\Pi$.

Having settled the full monoid  case $\tn$, we move onto one of its most important regular subsemigroups, namely that of singular transformations  $\sts$ . See \cite{eastsing2010} for a good overview regarding this semigroup. The ideal structure of $\sts$  was studied in detail inside the cross-connection theory in \cite{tx,tx1,tx2}. Naturally, our next objective is to use our theory of connected categories to realise the semigroup $\sts$ as a normal category. 

It is easy to see that the set of proper subsets of the set $\ns$ forms a small category $\pn$ with mappings as morphisms. We shall refer to $\pn$ as the {\it{singular powerset category}}. Observe that $\pn$ is a full subcategory of the full powerset category $\fpn$ (see Section \ref{ssfpn}). 

At this point, we call upon some known results.

\begin{lem}[{\cite[Theorem 3.1]{tx}\cite[Theorem 3]{tx1}}]
Let $T:=\sts$ be the singular transformation semigroup.  The category $\pn$ is normal and isomorphic to the left category $\mathbb{L}(T)$. The semigroup $\widehat{\pn}$ of all cones in the category $\pn$ is isomorphic to $T$, and so $T$ is totally left reductive.	
\end{lem}

Recall also that, since $\sts$ is a regular subsemigroup of $\tn$, the Green's relations in $\sts$ are inherited from $\tn$. Using Lemma \ref{lemgrtn} and the fact that $\sts$ is totally left reductive, we see that the poset of right ideals of the semigroup $\widehat{\mathbb{L}(\sts)}$ may be characterised using non-identity partitions of $\ns$.

Next, observe that the non-identity partitions of the set $\ns$ form a down-set, say $\pin$, of the poset $(\Pi,\supseteq)$. Also, given an element $\alpha\in \tn$, we have a principal cone $r^\alpha$ in the normal category $\mathbb{L}(\sts)$. This discussion allows us to  verify that the map  $R_{r^\alpha}\mapsto \pi_\alpha$ gives an order isomorphism from the poset of $\gr$-classes in the semigroup $\widehat{\mathbb{L}(\sts)}$ to the poset $\pin$. In fact, emulating the proof of Proposition \ref{proltn}, we get a description of  the semigroup of singular transformations  as a normal category. 

\begin{pro}
The normal category $\mathbb{L}(\sts)$ of the singular transformation semigroup is isomorphic to the singular powerset category $\pn$.  
\end{pro}

\subsection{Linear transformation semigroups}\label{ssecltv}
Continuing our list of applications, we move onto a brief discussion on the linear transformation semigroups, which are analogous to transformation semigroups. Given a finite dimensional vector space $V$ over a field $K$, the linear transformations on $V$ form a regular monoid $\tv$.  The group $\glv$ of invertible linear transformations on $V$ forms the subgroup of units of $\tv$. It is worth mentioning that the semigroup $M_n(K)$ of $n\times n$ matrices over $K$ may be realised as a special case of the semigroup $\tv$. 

In \cite{tlx}, the cross-connections of $\tv$, its singular part $\stv$ and the variants of $\tv$, were discussed. We refer the readers to \cite{tlx} for a detailed discussion on the normal categories involved. The right ideal structure of $\tv$ was described using the annihilator category in \cite{tlx}. However, employing our approach of connected categories, we can describe the right structure of $\tv$ with just the poset of null spaces of $V$.

Given a finite dimensional vector space $V$,  the subspaces of $V$ form a small category $\sv$ with linear transformations as morphisms. This category  $\sv$ has a largest object and the Green's relations $\gl$ and $\gr$  in $\tv$ are determined by subspaces and null spaces, respectively \cite[Section 2.2]{clif}. Let $\mathfrak{N}$ denote the poset of null spaces of $V$ under reverse inclusion. Imitating the proofs in the case of transformation semigroups, we obtain the next results.

\begin{pro} Let $V$ be a finite dimensional vector space  over a field $K$. The category
$\sv$ is  normal and bounded above. In particular, the poset of $\gr$-classes in the semigroup $\widehat{\sv}$ is isomorphic to the poset $\mathfrak{N}$, and so $\sv_\mathfrak{N}$ is a connected category. The normal category $\mathbb{L}(\tv)$ is isomorphic to $\sv$, as bounded above categories. 	
\end{pro}

Next, given the singular linear transformation semigroup  $\stv$, we can see that the proper subspaces of $V$ form a category. This category shall be called the \emph{singular subspace category} $\psv$. Clearly $\psv$ is a full subcategory of $\sv$. Using \cite[Theorem 2]{tlx} and the discussion of the semigroup of singular transformations,  we can also deal with the singular linear case.

\begin{pro}
Let $T:=\stv $ be the singular linear transformation semigroup. The category $\psv$ is  normal  and isomorphic to $\mathbb{L}(T)$, as normal categories. Moreover, the semigroup $\widehat{\psv}$ is isomorphic to the semigroup $T$.  
\end{pro}  

\subsection{Symmetric inverse monoids}
We conclude the paper with a quick discussion on arguably what is the most important inverse semigroup: the monoid of all the partial bijections on an arbitrary set $X$, denoted $\ix$. We shall realise $\ix$ as a self-supported category $\x$ which is also bounded above.

To begin with, recall that the Green's relations $\gl$ and $\gr$ in the semigroup $\ix$ are determined by the images and domains of the partial mappings, respectively \cite[Exercise 5.11.2]{howie}. Let $\x$ be the category of all subsets of $X$ with partial bijections as morphisms. We can easily verify that $\x$ is  normal and the set $X$ is the largest object in $\x$, whence $\x$ is bounded above. Now, define a functor $F\colon\mathbb{L}(\ix) \to \x$ as follows: for idempotents $\alpha,\beta\in \ix$ and $\theta\in \alpha\ix\beta$,
$$vF(\ix\alpha):= X\alpha\text{ and }F(r(\alpha,\theta,\beta)):=\theta.$$
We may check that $F$ is a normal category isomorphism. Hence the semigroup $\widehat{\x}$ of cones in $\x$ is isomorphic to the semigroup $\ix$. Moreover, since the $\gr$-classes of the semigroup $\ix$ are determined by the domains, the poset of $\gr$-classes of the semigroup $\widehat{\x}$ is order isomorphic to $v\x$. Hence the category $\x$ is self-supported too. We can verify that the isomorphism $F$ may be extended to an isomorphism of self-supported categories. We collect the above discussion in the final theorem, which describes how the symmetric inverse monoid may be realised as a category, just as in the ESN Theorem. 

\begin{thm}
Let $X$ be a set. The category $\x$ is  normal and  bounded above. The semigroup $\widehat{\x}$ of cones in $\x$ is isomorphic to the semigroup $\ix$. The normal category $\mathbb{L}(\ix)$ and the category $\x$ are isomorphic as self-supported categories.    
\end{thm}

Observe that the isomorphisms of $\x$ form an inductive groupoid $\mathcal{G}(\x)$ (see \cite[Section 3.4]{locinverse}), and ESN Theorem asserts that the inverse semigroup $\ix$ `is', in essence, the inductive groupoid $\mathcal{G}(\x)$.

\section*{Acknowledgement}
The authors wish to express their sincere gratitude to the anonymous referee for their exceptionally meticulous and thoughtful reading of the manuscript. The referee’s extensive comments and suggestions greatly improved the clarity of the exposition and helped to tighten several arguments and formulations throughout. Their careful engagement with the work is deeply appreciated.
\bibliography{azeefref}
\bibliographystyle{plain}

\end{document}